\documentclass[12pt,a4paper,oneside]{amsart}
\usepackage{amsmath,amssymb,amsthm,graphicx,mathrsfs,url}
\usepackage[usenames,dvipsnames]{color}
\usepackage[colorlinks=true,linkcolor=Red,citecolor=Green]{hyperref}
\usepackage{amsxtra}
\usepackage{wasysym} 
\usepackage{graphicx}
\usepackage{subcaption}

\usepackage{array}
\newcolumntype{P}[1]{>{\centering\arraybackslash}m{#1}}
\usepackage{biblatex}
\usepackage{blindtext}

\def\?[#1]{\textbf{[#1]}\marginpar{\Large{\textbf{??}}}}

\let\epsilon=\varepsilon 

\newtheorem{theo}{Theorem}
\newtheorem*{theo*}{Theorem}
\newtheorem{prop}{Proposition}[section]
\newtheorem*{prop*}{Proposition}	
\newtheorem{defi}[prop]{Definition}
\newtheorem*{defi*}{Definition}

\newtheorem{assumption}{Assumption}
\newtheorem{lemm}[prop]{Lemma}
\newtheorem*{lemm*}{Lemma}

\numberwithin{equation}{section}

\newtheorem{manualcorrreminner}{Corollary}
\newenvironment{manualcorr}[1]{%
  \renewcommand\themanualcorrreminner{#1}%
  \manualcorrreminner
}{\endmanualcorrreminner}

\DeclareMathOperator{\tr}{tr}

\usepackage{scalerel}

\newcommand\reallywidehat[1]{\arraycolsep=0pt\relax%
\begin{array}{c}
\stretchto{
  \scaleto{
    \scalerel*[\widthof{\ensuremath{#1}}]{\kern-.5pt\bigwedge\kern-.5pt}
    {\rule[-\textheight/2]{1ex}{\textheight}} 
  }{\textheight} %
}{0.5ex}\\           
#1\\                 
\rule{-1ex}{0ex}
\end{array}
}

\author{Tristan Humbert}
\email{humbertt@imj-prg.fr}
\address{Sorbonne Université, Paris France 75005.}

\begin{document}
\begin{abstract}
For a minimal Anosov $\mathbb R^{\kappa}$-action on a closed manifold, we study the measure of maximal entropy constructed by Carrasco and Rodriguez-Hertz in \cite{CarHer} and show that it fits into the theory of Ruelle-Taylor resonances introduced by Guedes Bonthonneau, Guillarmou, Hilgert, and Weich in \cite{GBGHW}. More precisely, we show that the topological entropy corresponds to the first Ruelle-Taylor resonance for the action on a certain bundle of forms and that the measure of maximal entropy can be retrieved as the distributional product of the corresponding resonant and co-resonant states. As a consequence, we prove a Bowen-type formula for the measure of maximal entropy and a counting result on the number of periodic torii.
\end{abstract}
\title{Measure of maximal entropy for minimal Anosov actions} 
\maketitle
\section{Introduction}
\subsection{Anosov actions}
Let $\mathcal M$ be a smooth closed (i.e compact and boundaryless) manifold equipped with a smooth Riemannian metric $g$. Consider  $\tau: \mathbb A\cong \mathbb R^{\kappa}\to \mathrm{Diffeo}^{\infty}(\mathcal M)$ a locally free action of an Abelian Lie group $\mathbb A$ of dimension $\kappa\geq 1$. Denote by $\mathfrak a:=\mathrm{Lie}(\mathbb A)\cong \mathbb R^{\kappa}$ its Lie algebra and define the \emph{infinitesimal action} by
\begin{equation}
\label{eq:XLie}
X:\begin{cases}
\mathfrak a \to C^{\infty}(\mathcal M;T\mathcal M)\\
A \mapsto X_A:=\frac d{dt}|_{t=0}\tau(\exp(At)),
\end{cases}
\end{equation}
where we write $\exp$ for the exponential map. We denote by $\varphi_t^A:=\tau(\exp(At))$ the flow at time $t\in \mathbb R$ corresponding to $A\in \mathfrak a$.
Since $\mathfrak a$ is Abelian, $\mathrm{Ran}(X)\subset C^{\infty}(\mathcal M;T\mathcal M)$ is a $\kappa$-dimensional subspace of commuting vector fields which spans a $\kappa$-dimensional subbundle $E_0\subset T\mathcal M$ which is called the \emph{central} or \emph{neutral} direction.
\begin{defi}
\label{Anosov}
An element $A\in \mathfrak a$ $($or equivalently $X_A)$ is \emph{transversely hyperbolic} if there is a continuous splitting of the tangent bundle
\begin{equation}
\label{eq:split}
T\mathcal M=E_s\oplus E_0\oplus E_u
\end{equation}
which is $d\varphi_t^A$-invariant $($i.e $\varphi_t^A(E_\bullet(x))=E_\bullet(\varphi_t^A x)$ for any $\bullet=s,u,0$, any $t\in  \mathbb R$  and  any $x\in \mathcal M)$ and there exist $($uniform$)$ constants $C,\nu>0$ such that
\begin{equation}
\label{eq:s}
\begin{split}
&\forall v\in E_s, \ \forall t\geq 0, \quad \|d\varphi_t^A(v)\|_g\leq Ce^{-\nu t}\|v\|_g,
\\
&\forall v\in E_u, \ \forall t\leq 0, \quad \|d\varphi_t^A(v)\|_g\leq Ce^{-\nu |t|}\|v\|_g.
\end{split}
\end{equation}
The action $\tau$ is \emph{Anosov} if there exists a transversely hyperbolic element $A\in \mathfrak a$. The distribution $E_s$ $($resp. $E_u)$ is the stable bundle $($resp. unstable bundle$)$ and its dimension will be denoted by $d_s$ $($resp. $d_u)$.
\end{defi}
\begin{defi}
The positive Weyl chamber $\mathcal W$ of $A_0$ is the set of $A\in \mathfrak a$ which are transversally hyperbolic with the same Anosov splitting. It is an open convex cone of~$\mathfrak a$.
\end{defi}
For $\kappa=1$, we recover the well-known definition of an Anosov flow.  For these flows, there are many invariant measures. A standard way of constructing an invariant measure is to consider the \emph{equilibrium state} associated to a real-valued and Hölder-continuous potential $V$ (see \cite[Theorem 4.3.13]{FishHas} for a precise definition). For the null-potential ($V=0$), we recover the \emph{measure of maximal entropy}.

For Anosov flows, powerful tools such as Markov partitions or the specification property lead to a very rich theory of {equilibrium states} see \cite{Sin68,Bo,Ru76,BoRu}. 

\subsection{Equilibrium states for partially hyperbolic flows}
Anosov actions of higher rank ($\kappa \geq 2$) are examples of partially hyperbolic flows for which the previously cited tools are not available. This makes the theory of equilibrium states much less developped in this case. Existence of a measure of maximal entropy can still be obtained by the upper-semi continuity of the entropy map $\mu\mapsto h(\mu,\varphi_1^A)$\footnote{Here, we denote by $h(\mu,\varphi_1^A$) the metric entropy of the time one map $\varphi_1^A$ with respect to the invariant measure $\mu$.} (see \cite{New} and the introduction of \cite{CPZ20} for an overview of the existing literature). However, this approach is non-constructive and thus does not give much information about the said measure.

Recently, geometrical constructions of equilibrium states were introduced by Climenhaga, Pesin and Zelerowicz \cite{CPZ19,CPZ20,Cli} and Carrasco and Rodriguez-Hertz \cite{CH1,CarHer} independently. The two approaches use different techniques but construct the same objects. Namely, a system $m^s$ (resp. $m^u$) of stable (resp. unstable) "leaf measures" whose product is the equilbrium state. Their constructions already provide new insights for Anosov flows \cite{CPZ19,CH1}. Moreover, unlike Markov partitions or the specification property, they extend to certain classes of partially hyperbolic flows \cite{CPZ20,CarHer} and in particular to Anosov actions of higher rank.

For an Anosov flow, the systems of leaf measures and the equilibrium state can also be constructed using a functional approach. More precisely, one can associate to $X$ (where $X$ is the generator of the flow) a discrete spectrum (the \emph{Ruelle resonances}) by making it act on specially designed functional spaces, see \cite{BKL,BL,BT,GL,Fau08,Fau10} for instance. The topological entropy $h_{\mathrm{top}}(\varphi_1)$ is a resonance called the \emph{first resonance}, i.e it is the real resonance with largest real part for the action on $d_s$-forms. Moreover, the system of stable (resp. unstable) leaf measures $m^s$ (resp. $m^u$) are eigenvectors corresponding to $h_{\mathrm{top}}(\varphi_1)$ as shown by Gouëzel-Liverani \cite[Theorem 5.1]{GouLi} for hyperbolic maps and by the author \cite[Theorem 1]{Hum} for Anosov flows.

A functional approach for general Anosov actions was developped by Guedes Bonthonneau, Guillarmou, Hilgert, and Weich in \cite{GBGHW}. This amounts to constructing a "good" joint spectral theory for the commuting vector fields $X_A$, for $A\in \mathcal W$ in the Weyl chamber, on some functional spaces (the so called \emph{anisotropic spaces}). In a companion paper \cite{GBGW}, they proved that the first Ruelle-Taylor resonance for the action on functions is $0$ and that the corresponding co-resonant states are invariant measures which have similar properties to the SRB measure in the classical rank one case. Moreover, if the action is positively transitive, they showed uniqueness of the SRB measure as well as full support. 

In the rank one case, one can study the equilibrium state associated to a potential $V+J^u$ where $V\in C^{\infty}(\mathcal M,\mathbb R)$ and $J^u$ is the unstable Jacobian by studying the first resonance of the operator $-X+V$ acting on functions, see \cite[Theorem 1]{Hum}. This provides a way to produce infinitely many invariant measures using the spectral approach. Note that in the higher rank case, the operators $X_A+V$ for $A\in \mathcal W$ do not commute so one cannot define their set of Ruelle-Taylor resonances. Thus it is not clear that one can produce many invariant measures using the spectral approach anymore. However, an important observation in the rank one case is that by making $X$ act on $d_s$-forms rather than functions, one can construct the measure of maximal entropy. We will follow this approach in this paper.

\subsection{Statement of results}
In the paper, we will work under the following hypothesis.
\begin{assumption}
\label{assumption}
Consider a smooth Anosov action $\tau: \mathbb A\to C^{\infty}(M;TM)$. Let $A_0\in \mathfrak a$ be a transversally hyperbolic element and suppose that its stable and unstable foliations are minimal, that is, each strong stable and strong unstable manifold is dense in $\mathcal M$. Suppose moreover that they are orientable.\footnote{One could dispose of the orientability condition by introducing a double cover. The minimality condition on the other hand seems to be at the core of the construction of Carrasco and Rodriguez-Hertz that we will use.}
\end{assumption}
 Working under Assumption \ref{assumption}, for any $A\in \mathcal W$, one can apply \cite[Corollary A]{CarHer} to the time-one map $\varphi_1^A$ and the null-potential.  This means that there exist two families of leaf measures $\{m^u_{A,x}\mid x\in \mathcal M, \ m^u_{A,x} \text{ measure on } \mathcal{W}^u(x)\}$  and $\{m^{cs}_{A,x}\mid x\in \mathcal M,\ m^{cs}_{A,x} \text{ measure on } \mathcal{W}^{cs}(x)\}$ whose product is equivalent to the measure of maximal entropy $m_A$ associated to $\varphi_1^A$. Here, $\mathcal W^{u}(x)$ and $\mathcal W^{cs}(x)$ denote the unstable and center stable manifold of $x$ respectively. Our first result states that the construction of Carrasco and Rodriguez-Hertz can be made uniform in the whole Weyl chamber $\mathcal W.$
\begin{theo}[Common measure of maximal entropy]
\label{Theo1}
There exist families of leaf measures $\{m^u_{x}\mid x\in \mathcal M\}$ and $\{m^{cs}_{x}\mid x\in \mathcal M\}$ such that 
\begin{equation}
\begin{split}
\label{eq:mu}
&\forall A\in \mathcal W,    \ (\varphi_1^A)^*(m^u_x)=e^{h_{\mathrm{top}}(\varphi_1^A)}m_{(\varphi_{1}^A)^{-1}x}^u,
\\&\forall A\in \mathcal W,  \ (\varphi_1^A)^*(m^{cs}_x)=e^{-h_{\mathrm{top}}(\varphi_1^A)}m_{(\varphi_{1}^A)^{-1}x}^{cs}.
\end{split}
\end{equation}
The product $m=cm^u\wedge m^{cs}$ $($for some normalizing constant $c>0)$ defines a probability measure which is the measure of maximal entropy associated to {any} $A\in \mathcal W$. It is invariant for any $\varphi_1^A$ for $A\in \mathbf{a}$. Moreovoer, it is ergodic and has the Bernoulli property with respect to any $A\in \mathcal W$. Finally, the entropy mapping\footnote{Here, $h_{\mathrm{top}}(\varphi_1^A)$ denotes the topological entropy of the time-one map $\varphi_1^A$.} $A\mapsto h_{\mathrm{top}}(\varphi_1^A)$   is linear in the Weyl chamber $\mathcal W$. 
\end{theo}
Applying Theorem \ref{Theo1} to $-X$, we obtain families of leaf measures $m^s$ and $m^{cu}$.
Next, we argue that $m^s$ and $m^u$ define resonant and co-resonant states associated to the first Ruelle-Taylor resonance for the action on the bundle of $d_s$-forms. Let 
\begin{equation}
\label{eq:dsbundle}
\mathscr E_0^{m}:=\{\omega \in C^{\infty}(\mathcal M;\Lambda^{m}T^*\mathcal M)\mid \forall A\in \mathfrak a, \ \iota_{X_A}\omega=0\},\quad 0\leq m \leq n-\kappa.
\end{equation}
Let $E_s^*,E_u^*, E_0^*\subset T^*\mathcal M$ be the dual bundles of the Anosov splitting \eqref{eq:s}:\begin{equation}
\label{eq:dual}
E_u^*(E_u\oplus E_0)=0, \ E_s^*(E_s\oplus E_0)=0, \ E_0^*(E_u\oplus E_s)=0,\quad T^*\mathcal M=E_u^*\oplus E_0^*\oplus E_s^*.
\end{equation}
The  system of measures $\{m^u_{x}\mid x\in \mathcal M\}$ defines a section $m^u$ of the dual of $\mathscr E_0^{d_s}$ (see \eqref{eq:muv} for the definition of the duality). We will call such a section a $d_s$-current and write $\mathcal D'(\mathcal M;\Lambda^{d_u}(E_s^*\oplus E_u^*))$ for the space of such currents.
Define $\mathbf{X}_A\omega:= \mathcal L_{X_A}\omega$. This is an admissible lift in the sense of \cite[Section 2.2]{GBGHW} and the theory of Ruelle-Taylor resonances is well defined by \cite[Theorem 4]{GBGHW}. Recall from \cite{GBGHW}  that $\lambda\in \mathfrak a_{\mathbb C}^*$ is a Ruelle-Taylor resonance if
\begin{equation}
\label{eq:vp1}
\exists u\in \mathcal D'(M;\Lambda^{d_s}(E_s^*\oplus E_u^*))\setminus\{0\}, \mathrm{WF}(u)\subset E_u^*,\ \forall A\in \mathcal W, \ \ -\mathbf{X}_Au=\lambda(A) u.
\end{equation}
Here, $\mathrm{WF}(u)\subset T^*\mathcal M\setminus \{0\}$ denotes the wavefront set of $u$, see \cite[Chapter 3]{Hor}. In this case, the current $u$ is called a \emph{resonant state} associated to the resonance $\lambda$. We have a dual notion of \emph{co-resonant state}:
\begin{equation}
\label{eq:cores}
\exists v\in \mathcal D'(M;\Lambda^{d_u}(E_s^*\oplus E_u^*))\setminus\{0\}, \mathrm{WF}(v)\subset E_s^*,\ \forall A\in \mathcal W, \ \ \mathbf{X}_Av=\lambda(A) v.
\end{equation}
Thanks to Theorem \ref{Theo1}, we can define $h_{\mathrm{top}}^{\mathcal W}\in \mathfrak a_\mathbb C^*$ such that $h_{\mathrm{top}}^{\mathcal W}(A)=h_{\mathrm{top}}(\varphi_1^A)$ for $A\in \mathcal W$ and extended by linearity on the rest of $\mathfrak a$.
\begin{theo}[First resonance]
\label{maintheo}
Let $\tau$ be an Anosov action on $(\mathcal M,g)$, a closed Riemannian manifold and suppose that $\tau$ satisfies Assumption \ref{assumption}.
Then one has
\begin{equation}
\label{eq:first res}
\forall A\in \mathfrak a,  \quad \begin{cases} \mathbf{X}_Am^u=h_{\mathrm{top}}^{\mathcal W}(A)m^u, \quad  \mathrm{WF}(m^u)\subset E_s^*
\\ -\mathbf{X}_Am^s=h_{\mathrm{top}}^{\mathcal W}(A)m^s, \quad  \mathrm{WF}(m^s)\subset E_u^*. \end{cases}
\end{equation}
In particular, $h_{\mathrm{top}}^{\mathcal W}$ is a Ruelle-Taylor resonance and $m^s$ $($resp. $m^u)$ is a corresponding resonant $($resp. co-resonant$)$  state. Moreover, the set of Ruelle-Taylor resonances is included in $\{\lambda \in \mathfrak a_\mathbb C^* \mid \mathrm{Re}(\lambda(A)) \leq  h_{\mathrm{top}}(\varphi_1^A),\ \forall A\in \mathcal W\}$. 
Finally, $h_{\mathrm{top}}^{\mathcal W}$ is the only resonance on the critical axis $\mathcal C:=\{\lambda \in \mathfrak a_\mathbb C^* \mid \mathrm{Re}(\lambda) = h_{\mathrm{top}}^{\mathcal W}  \}$ and it is simple, i.e it does not have Jordan block and the space of resonant $($resp. co-resonant$)$ states is one-dimensional and thus spanned by $m^s$ $($resp. $m^u)$. 
\end{theo}
This can be seen as a generalization of \cite[Theorem 1]{Hum} to the higher rank case or as an analog of \cite[Theorem 1]{GBGW} for the measure of maximal entropy. Similarly to the rank one case, the fact that the first resonance is simple essentially follows from the ergodicity of the measure of maximal entropy $m$. The absence of other resonances on the critical axis follows from the weak-mixing of $m$ which is implied by the stronger Bernoulli property satisfied by $m$, see \cite[Theorem C]{CarHer}.

The functional approach yields a Bowen-type formula for the measure of maximal entropy, that is, we obtain a formula for $m$ in terms of periodic orbits.
Recall that a point $x\in \mathcal M$ is called periodic if there is a $A\in \mathfrak a$ such that $\tau(A)x=x$. If $A\in \mathcal W$, it is known (see for instance \cite[Lemma 4.1]{GBGW}) that $T_x:=\{\tau(A')x\mid A'\in \mathfrak a\}$ is a $\kappa$-dimensional torus. The set of all periodic torii is denoted by $\mathcal T$ and for $T\in \mathcal T$, the associated lattice is denoted by  $L(T):=\{A'\in \mathfrak a\mid \tau(A')x=x\}$. For $X\subset \mathfrak a$, we denote by $|X|$ its volume and pushing forward the Lebesgue measure gives a measure $\lambda_T$ on each torus $T$. With these notations, our next result reads:
\begin{theo}[Bowen-type formula]
\label{theoBowen}
Under Assumption \ref{assumption} and let $\mathcal C$ be a proper sub-cone of the positive Weyl chamber $\mathcal W$. Let $\eta\in \mathfrak a^*$ be a dual element which is positive in a slightly larger open cone containing $\mathcal C$. For positive numbers $0<a<b$, define  $\mathcal C_{a,b}:=\{A\in \mathcal C\mid \eta(A)\in[a,b]\}$. Then for any $f\in C^{\infty}(\mathcal M)$, one has
\begin{equation}
\label{eq:bowen}
m(f)=\lim_{N\to +\infty}\frac{1}{|\mathcal C_{aN,bN}|}\sum_{T\in \mathcal T}\sum_{A\in \mathcal C_{aN,bN}\cap L(T)}e^{-h_{\mathrm{top}}(\varphi_1^A)}\int_T fd \lambda_T.
\end{equation}
\end{theo}
For Anosov flows, Bowen-type formulas for equilibrium states are usually obtained using the specification property. Hence, the extension of \eqref{eq:bowen} from the classical rank one case to the higher rank case is a priori non-trivial. Here, the use of the specification property is replaced by more analytical techniques and more precisely by the use of \emph{Guillemin's trace formula}, see Section \ref{Bowen type formula} for more details.

A similar formula was obtained in \cite[Theorem 4]{GBGW} for the SRB measure.
In the case of the Weyl chamber flow on a locally symmetric space $\mathcal M=\Gamma\setminus G/M$, the topological entropy map is given $\mathcal W\ni A\mapsto h_{\mathrm{top}}(\varphi_1^A)=2\rho(A)$ where $\rho$ is the half sum of positive roots. Note that the formula in this case was already obtained in \cite[Equation (0.3)]{GBGW} as in this special case, both the SRB measure and measure of maximal entropy coincide with the Haar measure.  In a recent paper \cite{Vin}, Vinhage constructed non-algebraic Anosov actions without rank one factor. In Appendix \ref{secappendix}, we show that his construction also provides examples of Anosov actions with no rank one factor for which the measure of maximal entropy and SRB measure are different. This gives further motivation to study the measure of maximal entropy in a general setting.

This result is interpreted as an equidistribution result of the periodic torii. As a consequence, we deduce the following corollary on the counting of periodic torii. Let
\begin{equation}
\|h_{\mathrm{top}}^{\mathcal W}\|:=\sup_{A\in \mathfrak a\setminus \{0\}}\frac{|h_{\mathrm{top}}^{\mathcal W}(A)|}{\|A\|}.
\end{equation}
\begin{manualcorr}{3.1}[Torii counting]
\label{corr}
Let $\mathcal C\subset \mathcal W$ be any proper subcone of the positive Weyl chamber. For $N>0$, define $\mathcal C_N:=\{A\in \mathcal C, \ h_{\mathrm{top}}(\varphi_1^A)/\|h_{\mathrm{top}}^{\mathcal W}\|\leq N\}$. Then one has
\begin{equation}
\label{eq:cor}
\lim_{N\to+\infty}\frac 1 N \ln \Big( \sum_{T\in \mathcal T}\sum_{A\in L(T)\cap \mathcal C_N}\mathrm{Vol}(T)\Big)=\|h_{\mathrm{top}}^{\mathcal W}\|.
\end{equation}
\end{manualcorr}
For an Anosov flow, periodic torii correspond to closed geodesics and $\|h_{\mathrm{top}}^{\mathcal W}\|$ is just the topological entropy of the flow $h_{\mathrm{top}}(\varphi_1)$. This means that \eqref{eq:cor} is a (weaker) logarithmic version of the Prime Orbit theorem \cite[Theorem 9.3]{PaPo} which holds for minimal Anosov actions of higher rank.

Other counting formulas on the number of periodic torii were obtained before, we cite \cite{Dei,Kni05,ELMV09,ELMV11,DaLi} and refer to the introductions of \cite{GBGW,DaLi} for a comparison of the different formulas. We remark that Dang and Li \cite[Theorems 1.2,1.3]{DaLi} give an exponentially small reminder in their equidistribution and counting results. Such a result could in theory be obtained for a general Anosov action if one could show the existence of a spectral gap in the Ruelle-Taylor resonances. Nevertheless, it is not clear under what assumption such a gap could be obtained.

We note however that all previously cited works were done in the case of Weyl chamber flows on locally symmetric spaces. The first counting result valid for general Anosov actions was proven in
 \cite[Corollary 0.4]{GBGW} under some asymptotic assumption on the Poincaré determinant. The previous corollary is thus a generalization of \cite[Corollary 0.4]{GBGW} where no assumption on the Poincaré determinant is needed.

\textbf{Acknowledgements.} The author would like to first thank Colin Guillarmou and Thibault Lefeuvre for their guidance and advice during the writing of this paper.

The author would also like to thank Yannick Guedes Bonthonneau and Tobias Weich for comments on an earlier version of the paper and Tobias Weich for suggesting the problem. Finally, he would like to thank Kurt Vinhage and Alp Uzman for interesting discussions on Anosov actions and Pablo Carrasco for suggesting the proof of Proposition \ref{propinv}.

This research was supported by the European Research Council (ERC) under
the European Union’s Horizon 2020 research and innovation programme (Grant agreement no. 101162990 — ADG).

\section{Measure of maximal entropy for Anosov actions}
In this section, we review the construction of the measure of maximal entropy of Carrasco and Rodriguez-Hertz and show Theorem \ref{Theo1}. Fix a homeomorphism $f:X\to X$ of a compact metric space $X$ and let $\mathcal P_f(X)$ be the set of $f$-invariant Borel probability measures on $X$.  Recall the variational principle which states that
\begin{equation}
\label{eq:varia}
h_{\mathrm{top}}(f)=\sup_{\mu \in \mathcal P_f(M)}h(f,\mu),
\end{equation}
where $h_{\mathrm{top}}(f)$ is the topological entropy of $f$ and $h (f, \mu)$ is the metric entropy with respect to $\mu$. An invariant probability measure $\mu$ is a measure of maximal entropy (or an equilibrium state for the null potential) if $h(f,\mu)=h_{\mathrm{top}}(f)$.

In the following, we consider an Anosov action $\tau :\mathbb A\to C^{\infty}(\mathcal M;T\mathcal M)$ and fix a transversally hyperbolic element $A_0$ as well as its positive Weyl chamber $\mathcal W$. We further suppose that the unstable and stable foliations are minimal which allows us to use \cite[Corollary A]{CarHer}. Before that, we need to introduce some terminology.

Both geometric approaches of Climenhaga et al and Carrasco and Rodriguez-Hertz start by constructing leaf measures and then deduce the construction of the equilibrium state by a product construction. For an Anosov action, all bundles $E_s,E_u,E_{cs}:=E_0\oplus E_s, E_{cu}:=E_0\oplus E_u$ from Definition \ref{Anosov} are integrable to Hölder continuous foliations denoted by $\mathcal W^s,\mathcal W^u,\mathcal W^{cs},\mathcal W^{cu}$ respectively and the leaves of the foliation are smooth, see \cite[Theorem 6.2.8 and \S 6.4]{KaHas}. These foliations are called the stable, unstable, center stable and center unstable foliation respectively. In the following, a system of leaf measure will be an element of
\begin{equation}
\label{eq:bundle}
\mathrm{Meas}^{\bullet}:=\{\nu: [x]\in (\mathcal M/\sim_{\bullet})\mapsto \nu_x \in \mathrm{Rad}(\mathcal W^{\bullet}(x))\},
\end{equation}
where $\bullet=s,u,cs,cu$, $(\mathcal M/\sim_{\bullet})$ is the quotient of $\mathcal M$ by the equivalence relation defined by $x\sim_{\bullet} y \iff \mathcal W^{\bullet}(x)=\mathcal W^{\bullet}(y)$ and $\mathrm{Rad}(X) $ denotes the set of Radon measures on~$X$. In other words, a system of  $\bullet-$leaf measures is the data of $\{m^{\bullet}_x\mid x\in \mathcal M\} $ where $m^{\bullet}_x$ is a Radon measure on a $\bullet-$manifold $\mathcal W^{\bullet}(x)$ satisfying the following compatibility condition. For any $x,x'\in \mathcal M$ such that $x\sim_{\bullet}x'$, one has $m^{\bullet}_x=m^{\bullet}_{x'}$. We can also define
\begin{equation}
\label{eq:Con}
\begin{split}
\mathrm{Con}^{\bullet}&:=\{f: [x]\in (\mathcal M/\sim_{\bullet})\mapsto f_x \in \mathrm{Con}(\mathcal W^{\bullet}(x))\},
\\
(\mathrm{Con}^+)^{\bullet}&:=\{f: [x]\in (\mathcal M/\sim_{\bullet})\mapsto f_x \in \mathrm{Con}^+(\mathcal W^{\bullet}(x))\}
\end{split}
\end{equation}
where $\mathrm{Con}(X)$ denotes the set of compactly supported smooth functions on $X$ and $\mathrm{Con}^+(X)$ denotes the set of non-negative compactly supported smooth functions on $X$. In the following, we might drop the index $\bullet$ if it is clear from the context which one we refer to. Note that $\mathrm{Meas}^{\bullet}$ is naturally endowed with the weak topology induced by  $\mathrm{Con}^{\bullet}$. We now state \cite[Corollary A, Theorem C]{CarHer} in the special case of Anosov actions.

For any $A\in \mathcal W$, there exists $m_A\in \mathcal P_{\varphi_1^A}(\mathcal M)$ and families of leaf measures $m_A^{\bullet}=\{m^{\bullet}_{x,A}\mid x\in \mathcal M\}$ where $\bullet=u,s,cu,cs$ such that :
\begin{enumerate}
\item the measure $m_A$ is the \emph{unique} measure of maximal entropy (MME) for the partially hyperbolic dynamical system $(\mathcal M,\varphi_1^A)$.
\item The measure $m^{\bullet}_x$ for any $ x\in \mathcal M$ is positive on relatively open sets, that is, it has full support in each leaf.
\item For any $x\in \mathcal M$, one has
\begin{align*}
m^{\bullet}_{\varphi_1^A x,A}&	=e^{h_{\mathrm{top}}(\varphi_1^A)}(\varphi_1^A)_*m^\bullet_{x,A}, \ \bullet\in \{u,cu\}
\\
m^{\bullet}_{\varphi_1^Ax,A}&	=e^{-h_{\mathrm{top}}(\varphi_1^A)}(\varphi_1^A)_*m^\bullet_{x,A}, \ \bullet\in \{s,cs\}.
\end{align*}
\item For any measurable partition $\xi$ which refines the partition by unstable (resp. stable) manifolds, the conditional measures of $m_A$ are equivalent ($m_A$ a.e) to the leaf measures. Moreover, $(m_A)_{|B(x,\delta)}$ is equivalent to the product measures $m^u_{x,A}\times m^{cs}_{x,A}$ and $m^s_{x,A}\times m^{cu}_{x,A}$ for any $x\in \mathcal M$ and $\delta>0$ small enough.
\item The measure of maximal entropy $m_A$ satisfies the \emph{Gibbs property}. For any $\epsilon>0$, there exists $A,B>0$ such that
$$\forall x\in \mathcal M, \ \forall n\geq 0, \quad A\leq \frac{m_A(B_n(x,\epsilon))}{e^{-nh_{\mathrm{top}}(\varphi_1^A)}}\leq B $$
where $B_n(x,\epsilon)$ is the Bowen ball, defined by
\begin{equation}\label{eq:BB}
B_n(x,\epsilon):=\{y\in \mathcal M \mid \max_{k=0}^nd(\varphi_k^A x,\varphi_k^A y)<\epsilon\}. 
\end{equation}
\item The measure $m$ has the Bernoulli property.
\end{enumerate}
We note that a similar construction of a system of leaf measures was obtained by Buzzi, Fisher and Tazhibi in \cite{BFT}.
In the rest of the section, we will prove that the construction of $m^\bullet$ can be made independent of $A$ in the Weyl chamber $\mathcal W$.
\begin{prop}
\label{uniform}
There exist families of leaf measures $m^{\bullet}=\{m^{\bullet}_{x}\mid x\in \mathcal M\}$ where $\bullet=u,s,cu,cs$ such that for any $x\in \mathcal M$ and any $A\in \mathcal W$
\begin{equation}
\label{eq:phit}
\begin{split}
&m^{\bullet}_{\varphi_1^A x}	=e^{h_{\mathrm{top}}(\varphi_1^A)}(\varphi_1^A)_*m^\bullet_{x}, \quad \bullet\in \{u,cu\}
\\
&m^{\bullet}_{\varphi_1^Ax}	=e^{-h_{\mathrm{top}}(\varphi_1^A)}(\varphi_1^A)_*m^\bullet_{x}, \quad \bullet\in \{s,cs\}.
\end{split}
\end{equation}
As a consequence, the measure of maximal entropy $m$ is common to all $A\in \mathcal W$ and the entropy mapping $A\mapsto h_{\mathrm{top}}(\varphi_1^A)$ is linear in the Weyl chamber $\mathcal W$.
\end{prop}
\begin{proof}
We adapt slightly the arguments of \cite{CarHer}. Note that from \cite[Equation $(12)$]{CarHer}, the system of measures $m^u_x$ is constructed from $m^{cu}_x$ by taking the pushforward by the projection $\pi_x^c: \mathcal W^c(\mathcal W^u(x),r)\to \mathcal W^u(x)$ (for $r>0$ small enough) by sliding along the local center plaque. This means that it suffices to show that the system of weak-unstable leaf measures can be constructed uniformly in the Weyl chamber. First, notice that from \cite[Theorem A, $(4)$]{CarHer} and the ergodicity of the equilibrium measure, one sees that for a fixed $A\in \mathcal W$, the system of weak-unstable leaf measures $m^{cu}_{x,A}$ is unique up to a (global) constant rescaling. In particular, if we fix a section $\phi_0\in \mathrm{Con}^+$, it suffices to show that given any two $A_1,A_2\in \mathcal W$, one has $m^{cu}_{x,A_1}=m^{cu}_{x,A_2}$ under the normalization condition $m^{cu}_{x,A_1}(\phi_0)=m^{cu}_{x,A_2}(\phi_0)$. 
We recall the following lemma from \cite[Lemma 2.1]{CarHer}.
\begin{lemm}
\label{lemm:CarHer}
Let $\mathcal A\subset \mathrm{Meas}^{cu}$ be such that
\begin{itemize}
\item For any $\phi\in \mathrm{Con}$, there exists  a constant $c(\phi)>0$ such that for any $ \mu \in \mathcal A$, one has $\mu(\phi)\leq c(\phi)$.
\item For any $\phi\in \mathrm{Con}^+$, there exists  a constant $c'(\phi)>0$ such that for any $ \mu \in \mathcal A$, one has $\mu(\phi)\geq c'(\phi)$.
\end{itemize}
Then $\mathcal A$ is compact and does not contain the zero section.
\end{lemm}
The starting point of their argument consists in finding a reference measure $\nu$ which is \emph{appropriate}. This means the following.
\begin{itemize}
\item  It is strongly absolutely continuous, i.e there is $\delta_0>0$ and a continuous map $J:\{(x,y,z)\mid x\in M, \ y\in \mathcal W^s(x,\delta_0), z\in \mathcal W^{cu}(y,\delta_0)\}\to \mathbb R$ such that 
$$ (\mathrm{Hol}_{s}^\delta)_*\nu_x=J(x,y, \cdot) \nu_y$$
where $\mathrm{Hol}_s^\delta$ is the holonomy transport along local $\delta$-stable leaves, see \cite[Definition 2.3]{CarHer}.
\item It has full support in each weak-unstable leaf.
\item It is quasi-invariant with Hölder continuous Jacobian with respect to $\varphi^{A_i}_1$ for $i=1,2$. This means that the pushforward measure $(\varphi_1^{A_i})_*\nu_x$ is equivalent to $\nu_{\varphi_1^{A_i}(x)}$ with Hölder continuous density. The density will be referred to as the Jacobian.
\end{itemize}
We will use \cite[Proposition 2.3]{CarHer} which states that the restriction of the Lebesgue measure on each weak-unstable leaf is strongly absolutely continuous. We will write $\mathrm{Leb}$ for this system of measures. It is clear that it has full support in each leaf. Note that the last condition, the quasi-invariance, depends on the partially hyperbolic diffeomorphism (while the first two only depend on the Anosov splitting) and we will thus be a little more explicit about~it.

For $i=1,2$, the pushforward measure $(\varphi_1^{A_i})_*\mathrm{Leb}|_{\mathcal W^{cu}(x)}$ is equivalent to $\mathrm{Leb}|_{\mathcal W^{cu}(\varphi_1^{A_i}(x))}$. Moreover, the associated Jacobian is given by $x\mapsto |\mathrm{det}(d\varphi_1^{A_i}|_{E_u(x)})|$ and is Hölder continuous. This means that $\mathrm{Leb}$ is an appropriate measure with respect to $\varphi_1^{A_i}$ for $i=1,2$.

In the following, we will fix $(C_0,\alpha)$ such that $x\mapsto \ln|\mathrm{det}(d\varphi_1^{A_i}|_{E_u(x)})|$ for $i=1,2$ are both $(C_0,\alpha)$-Hölder continuous. That is, for any $x,y\in \mathcal M$ and $i=1,2$, one has
\begin{equation}
\label{eq:Holder} \big|\ln|\mathrm{det}(d\varphi_1^{A_i}|_{E_u(x)})|-\ln|\mathrm{det}(d\varphi_1^{A_i}|_{E_u(y)})| \big|\leq C_0 d(x,y)^{\alpha}.
\end{equation}
We fix $\phi_0\in \mathrm{Con}^+$ and define the following set
\begin{equation}
\label{eq:Chi}
\mathcal X:=\overline{\mathrm{Conv}\Big\{\nu^{n,m}:=\frac{(\varphi_1^{nA_1+mA_2})_*\mathrm{Leb}}{(\varphi_1^{nA_1+mA_2})_*\mathrm{Leb}(\phi_0)}\mid (n,m)\in \mathbb N\times \mathbb N\Big\}}\subset \mathrm{Meas}^{cu}
\end{equation}
where $\mathrm{Conv}(X)$ denotes the convex hull of a set $X$ and $\overline X$ its closure.

\textbf{The set $\mathcal X$ is compact.}
We show that $\mathcal X$ satisfies the hypothesis of Lemma \ref{lemm:CarHer} and is hence compact. We will then construct the system of leaf measure $m^{cu}$ as a fixed point using the Schauder-Tychonoff fix point theorem. We prove the following :
\begin{lemm}
\label{lemm1}
Let $\psi \in \mathrm{Con}^+$ and $\phi \in \mathrm{Con}$. Then there exists $e(\psi,\phi)>0$ such that \begin{equation}
\label{eq:compact}
\forall (n,m)\in \mathbb N^2, \quad \frac{\nu^{n,m}(\phi)}{\nu^{n,m}(\psi)} \leq e(\psi,\phi).
\end{equation}
As a consequence, $\mathcal X\subset \mathrm{Meas}$ is a compact subset.
\end{lemm}
\begin{proof}
This corresponds to a slight adaptation of the combination of \cite[Lemma 2.5, Lemma 2.6 and Corollary 2.7]{CarHer}. We recall that two sections $\psi_1,\psi_2 \in \mathrm{Con}$ are said to be $\delta$-equivalent if $\mathrm{Supp}(\psi_1)$ and $\mathrm{Supp}(\psi_2)$ are homeomorphic  via $\mathrm{Hol}_s^\delta$ and for any $x\in \mathrm{Supp}(\psi_1)$, one has $\psi_2(\mathrm{Hol}^{\delta}_sx)=\psi_1(x)$. We will fix a $\delta>0$ and drop it in the notation for the following computations. 
\\
We first show that there exists $D_1>0$ and $\ell :\mathbb R_+\to \mathbb R_+$ such that for any $\delta>0$, any $\psi_1,\psi_2 \in \mathrm{Con}^+$ which are $\delta$-equivalent and any $(n,m)\in \mathbb N^2$, 
\begin{equation}
\label{eq:techn}
\nu^{n,m}(\psi_1)\leq \ell(\delta)\cdot e^{D_1\delta^{\alpha}}\nu^{n,m}(\psi_2),
\end{equation}
where $\alpha\in \mathbb R_+$ is defined in \eqref{eq:Holder}.
We note that $\mathcal C:=\{t_1A_1+t_2A_2\mid t_1,t_2\geq 0\}\subset \mathcal W$ is a proper subcone of the Weyl chamber. In particular, since the constants appearing in the Anosov property \eqref{eq:s} can only diverge near the boundary of $\mathcal W$, one can find uniform Anosov constants $C_1, \nu>0$ for the whole cone $\mathcal C$.

We will write $h_{n,m}=\tfrac{d \nu^{n,m}}{d\mathrm{Leb}}$ for the density of $\nu^{n,m}$. Now, we compute for $\psi_1,\psi_2 \in \mathrm{Con}^+$ which are $\delta$-equivalent and any $(n,m)\in \mathbb N^2$, 
\begin{align*}\nu^{n,m}(\psi_1)&=\int_{\mathcal M} \psi_1(x) h_{n,m}(x)d\mathrm{Leb}(x)=\int_{\mathcal M} (\psi_2\circ \mathrm{Hol}_s(x)) h_{n,m}(x)d\mathrm{Leb}(x)
\\&=\int_{\mathcal M} \frac{h_{n,m}(x)}{h_{n,m}(\mathrm{Hol}_s(x))}(\psi_2\circ \mathrm{Hol}_s(x))h_{n,m}(\mathrm{Hol}_s(x)){d\mathrm{Leb}}(x)
\\&\leq \sup_x \left|\frac{h_{n,m}(x)}{h_{n,m}(\mathrm{Hol}_s(x))}\right|\int_{\mathcal M}\psi_2(x) h_{n,m}(x)d((\mathrm{Hol}_s)_*\mathrm{Leb})(x)
\end{align*}
which then gives 
\begin{equation}
\label{eq:ineq}
\nu^{n,m}(\psi_1)\leq \sup_x \left|\frac{h_{n,m}(x)}{h_{n,m}(\mathrm{Hol}_s(x))}\right| \times \sup_x\left|\frac{d((\mathrm{Hol}_s)_*\mathrm{Leb})(x)}{d\mathrm{Leb}(x)}\right|\times \nu^{n,m}(\psi_2).
\end{equation}
We first consider $\ell$ to be an upper bound of $\tfrac{d((\mathrm{Hol}_s)_*\mathrm{Leb})(x)}{d\mathrm{Leb}(x)}$. Next, we use the chain rule, the Hölder continuity \eqref{eq:Holder} as well as the Anosov property \eqref{eq:s} to obtain
\begin{align*}\left|\frac{h_{n,m}(x)}{h_{n,m}(\mathrm{Hol}_sx)}\right|&\leq  \prod_{i=1}^n\frac{|\mathrm{det}(d\varphi_1^{A_1}|_{E_u(\varphi_{1}^{(i-1)A_1}x)})|}{|\mathrm{det}(d\varphi_1^{A_1}|_{E_u(\varphi_{1}^{(i-1)A_1}\mathrm{Hol}_s(x))})|} \prod_{j=1}^m\frac{|\mathrm{det}(d\varphi_1^{A_2}|_{E_u(\varphi_{1}^{nA_1+(j-1)A_2}x)})|}{|\mathrm{det}(d\varphi_1^{A_2}|_{E_u(\varphi_{1}^{nA_1+(j-1)A_2}\mathrm{Hol}_s(x))})|} 
\\&\leq\exp\left(C_0\sum_{i=1}^nd\big(\varphi_{1}^{(i-1)A_1}x,\varphi_{1}^{(i-1)A_1}\mathrm{Hol}_s(x)\big)^\alpha \right)
\\&\times\exp\left(C_0\sum_{j=1}^md\big(\varphi_{1}^{nA_1+(j-1)A_2}x,\varphi_{1}^{nA_1+(j-1)A_2}\mathrm{Hol}_s(x)\big)^\alpha \right)
\\&\leq \exp\left(C_0\sum_{n\geq 0}C_1^{\alpha}e^{-\eta \nu \alpha n}\delta^{\alpha} \right)=:\exp(D_1\delta^{\alpha}).
\end{align*}
Plugging this last estimate into \eqref{eq:ineq} yields \eqref{eq:techn}. Now, we follow the proof of \cite[Lemma 2.6]{CarHer}. Given $x_1,x_2\in \mathcal M$, $X_1\subset \mathcal W^{cu}(x_1)$ and $X_2\subset \mathcal W^{cu}(x_2)$ two open and pre-compact sets, we show that there is a constant $\hat e(X_1,X_2)>0$ such that
\begin{equation}
\label{eq:nn}
\forall n,m\geq 0,\quad \frac{1}{\hat e(X_1,X_2)}\leq \frac{\nu^{n,m}(X_1)}{\nu^{n,m}(X_2)}\leq \hat e(X_1,X_2).
\end{equation}
We first use Assumption \ref{assumption} and more precisely that the stable foliation is minimal to deduce the following property. For any $x_0\in \mathcal M$ and any $A\subset \mathcal W^{cu}(x_0)$, there are $\delta(A),r(A)>0$ such that for any $x\in \mathcal M$, one can find $B_x\subset A$ which is $\delta(A)-$equivalent to $\mathcal W^{cu}(x,r(A))$. Here, $\mathcal W^{cu}(x,r(A))=\{y\in \mathcal W^{cu}(x)\mid d_{\mathcal W^{cu}(x)}(x,y)<r(A)\}$ denotes the local center-stable manifold and the distance is the one induced by the Riemannian metric on the leaves. Using the relative compactness of $X_1$, one can write $\overline{X_1}\subset \cup_{i=1}^m\mathcal W^{cu}(x_j,r(X_2))$ where each $\mathcal W^{cu}(x_j,r(X_2))$ is $\delta(X_2)$-equivalent to some $B_j\subset X_2$. Approximating characteristic functions by smooth functions and using \eqref{eq:techn} yields for any $j=1, \ldots, m$ (see \cite[Lemma 2.6]{CarHer} for the details):
$$\frac{\nu^{n,m}(\mathcal W^{cu}(x_j,r(X_2)))}{\nu^{n,m}(B_j)}\leq \ell(\delta(X_2))e^{D_1	\delta(X_2)^\alpha}:=M. $$
This finally gives
$$ \frac{\nu^{n,m}(X_1)}{\nu^{n,m}(X_2)}\leq \sum_{j=1}^m\frac{\nu^{n,m}(\mathcal W^{cu}(x_j,r(X_2)))}{\nu^{n,m}(X_2)}\leq m\max_{j=1}^n\frac{\nu^{n,m}(\mathcal W^{cu}(x_j,r(X_2)))}{\nu^{n,m}(B_j)}\leq mM.$$
Eventually, we are able to deduce \eqref{eq:nn} by exchanging the roles of $X_1$ and $X_2$. We now prove Lemma \ref{lemm1}. Since $\psi$ is non-negative and non identically zero, there is a $r>0$ such that $A_r:=\psi^{-1}(r,+\infty)$ is relatively open and pre-compact. Choose $A$ open and relatively compact containing the support of $\phi$, then using \eqref{eq:nn} for $A_r$ and $A$, we obtain
$$\frac{\nu^{n,m}(\phi)}{\nu^{n,m}(\psi)} \leq \frac{\|\phi\|_{\infty}\nu^{n,m}(A)}{r\nu^{n,m}(A_r)}\leq \frac{\|\phi\|_{\infty}}{r}\hat e(A,A_r).  $$

 We now show that $\mathcal X$ is compact, for this, fix  any $\phi\in \mathrm{Con}$. Then apply \eqref{eq:compact} with $\psi=\phi_0$ to obtain the first condition of Lemma \ref{lemm:CarHer}. Now fix any  $\eta\in \mathrm{Con}^+$ and apply \eqref{eq:compact} with $\phi=\phi_0$ and $\psi=\eta$ to obtain the second condition of Lemma \ref{lemm:CarHer} and the compactness of $\mathcal X$.
\end{proof}
\textbf{Constructing the common system.}
 Consider the following continuous mapping
$$S: \mathcal X \to \mathcal X, \quad S(\eta):=\frac{(\varphi_{1}^A)_*\eta}{(\varphi_{1}^A)_*\eta(\phi_0)}.$$
We see that $\mathcal X$ is invariant under $S$ so by the Schauder-Tychonoff fix point theorem, $S$ has a fix point $\mu$:
\begin{equation}
\label{eq:fixpoint}
\exists \mu \in \mathcal X, \ S(\mu)=\mu \iff (\varphi_{1}^{A_1})_*\mu=e^{\lambda}\mu, \ \lambda \in \mathbb R.
\end{equation}
Now, $\mu$ is quasi-invariant with Jacobian $e^\lambda$ and is in the closure of the positive cone generated by $\{\nu^{n,m}\}_{n,m\geq0}$ (thus has full support in each leaf). This means one can use \eqref{eq:techn} to adapt the proof of \cite[Lemma 2.8]{CarHer} and for any $\epsilon>0$, there is a $\delta>0$ such that for any $\psi_1,\psi_2\in \mathrm{Con}^+$ that are $\delta$-equivalent, one has $|\mu(\psi_1)/\mu(\psi_2)-1|<\epsilon. $ This is the only thing we need to adapt the proof of \cite[Proposition 2.9]{CarHer} which shows that $\mu$ is strongly absolutely continuous. In other words, the measure $\mu$ is appropriate for $\varphi_1^{A_i}$ for $i=1,2$. 
This in turn means that one can apply the results of \cite[Section $3$]{CarHer} to $\mu$. In particular, using \cite[Proposition 3.19]{CarHer}, we see that the Jacobian is actually given by $\lambda=h_{\mathrm{top}}(\varphi_1^{A_1})$. To summarize, starting from an appropriate measure, one can construct another appropriate measure which is quasi-invariant with Jacobian given by the exponential of the topological entropy of the partially hyperbolic diffeomorphism.

We then consider the following subset of $\mathcal X$
\begin{equation}
\label{eq:Y}
\mathcal Y:=\overline{\mathrm{Conv}\Big\{\alpha^n:=\frac{(\varphi_n^{A_2})_*\mu}{(\varphi_n^{A_2})_*\mu(\phi_0)}\mid n\in \mathbb N\Big\}}\subset \mathcal X.
\end{equation}
The space $\mathcal Y$ is compact as a closed subset of $\mathcal X$. Since $\mu$ was shown to be appropriate, reapplying the argument above (or directly \cite[Section 2]{CarHer}) shows that 
\begin{equation}
\label{eq:fixpoint2}
\exists \beta \in \mathcal Y, \  (\varphi_{1}^{A_2})_*\beta=e^{h_{\mathrm{top}}(\varphi_1^{A_2})}\beta.
\end{equation}
However, note that any element $\theta$ of $\mathcal Y$ satisfies $(\varphi_{1}^{A_1})_*\theta=e^{h_{\mathrm{top}}(\varphi_1^{A_1})}\theta$ since the two flows commute. This means that  $\beta$ is an appropriate measure such that
$$(\varphi_{1}^{A_1})_*\beta=e^{h_{\mathrm{top}}(\varphi_1^{A_1})}\beta, \quad (\varphi_{1}^{A_2})_*\beta=e^{h_{\mathrm{top}}(\varphi_1^{A_2})}\beta. $$
Using  \cite[Theorem A, $(4)$]{CarHer}, we see that any system of measures $m^{cu}_{A_1}$ (resp. $m^{cu}_{A_2}$) is obtained as some constant rescalling of $\beta$.
In other words, we have shown that one has $m^{cu}_{A_1}=\beta=m^{cu}_{A_2}$ for any two $A_1,A_2\in \mathcal W$ if the leaf measures are normalized such that $m^{cu}_{A_i}(\phi_0)=1$ for $i=1,2$. This concludes the proof of \eqref{eq:phit}.

\textbf{The topological entropy is linear in the Weyl chamber.}
We deduce that the entropy mapping $A\mapsto h_{\mathrm{top}}(\varphi_1^A)$ is linear in the Weyl chamber $\mathcal W$. Indeed, let $A_1,A_2 \in \mathcal W$ and $\lambda_1,\lambda_2\geq 0$, then 
\begin{align*}
(\varphi_t^{\lambda_1A_1+\lambda_2A_2})_*m^{cu}&=e^{th_{\mathrm{top}}(\varphi_1^{\lambda_1A_1+\lambda_2A_2})}m^{cu}=(\varphi_t^{\lambda_1A_1})_*(\varphi_t^{\lambda_2A_2})_*m^{cu}
\\&=e^{t(\lambda_1 h_{\mathrm{top}}(\varphi_1^{A_1})+\lambda_2 h_{\mathrm{top}}(\varphi_1^{A_2}))}m^{cu}.
\end{align*}
And thus, we obtain $h_{\mathrm{top}}(\varphi_1^{\lambda_1A_1+\lambda_2A_2})=\lambda_1 h_{\mathrm{top}}(\varphi_1^{A_1})+\lambda_2 h_{\mathrm{top}}(\varphi_1^{A_2}).$
\end{proof}
In the rest of the section, we show that the common measure of maximal entropy $m$ is invariant under any $\varphi^1_A$ for $A\in \mathbf a$. The authors would like to thank Pablo Carrasco for pointing this to him.
\begin{prop}
\label{propinv}
The measure $m$ constructed in Proposition \ref{uniform} satisfies
\begin{equation}
    \label{eq:invariance}
    \forall A\in \mathbf a, \quad (\varphi_1^A)_*m=m.
\end{equation}
\end{prop}
\begin{proof}
    Let $A\in\mathbf a $. Since the action is Abelian, we see that $\varphi^{A_0}_1$ permutes the unstable leaves, i.e, 
    $$\forall x\in \mathcal M,\quad \varphi^{A_0}_1\left(\mathcal W^u(x) \right)=\mathcal W^u(\varphi^{A_0}_1(x)). $$
    This means that if $m^u$ is the system of unstable measures constructed in Proposition \ref{uniform}, $\tilde m^u:=(\varphi^1_A)_*m^u$ is a system of unstable measures which satisfies
    or any $x\in \mathcal M$ and any $A\in \mathcal W$
\begin{equation}
\label{eq:newmu}
(\varphi_1^{A})_*\tilde m^u_{x}=(\varphi_1^{A_0})_*(\varphi_1^{A})_*\tilde m^u_x=e^{-h_{\mathrm{top}}(\varphi_A^1)}\tilde m^u_{\varphi_1^Ax},   \end{equation}
where we used that the flows commute and \eqref{eq:phit}. Since $\tilde m^u$ is also fully supported in each unstable leaf, \cite[Theorem A, $(4)$]{CarHer} and the ergodicity of the equilibrium measure imply that there exists $c_u>0$ such that $\tilde m^u=(\varphi^1_A)_*m^u=c_um^u$. Similarly, there is a constant $c_s>0$ such that $(\varphi_1^{A_0})_*m^{cs}=c_sm^{cs}$. Now, this means that $$(\varphi_1^{A_0})_*m=c_uc_s(m^u \wedge m^{cs})=c_uc_sm.$$
Using that $\varphi_1^{A_0}(\mathcal M)=\mathcal M$ and the fact that $m$ is a probability measure finally gives
$1=[(\varphi_1^{A_0})_*m](\mathcal M)=c_uc_s m(\mathcal M)=c_uc_s. $ This shows that $m$ is invariant by $\varphi_1^{A_0}.$
\end{proof}
\section{Ruelle-Taylor resonances for the action on $d_s$-forms}
\subsection{Ruelle-Taylor resonances}
In this subsection, we recall the main features of Ruelle-Taylor resonances. They are a generalization of Ruelle resonances to the higher rank case and were first introduced by Guedes Bonthonneau, Guillarmou, Hilgert and Weich in \cite{GBGHW}. We refer to this paper for details on the construction as we will only state the important properties needed for our work. First, we consider the bundle of $d_s$-form in the kernel of the contraction by the center direction:
\begin{equation}
\label{eq:E_0^m}
\mathscr E_0^{d_s}:=\{\omega \in C^{\infty}(\mathcal M; \Lambda^{d_s}T^*\mathcal M)\mid \iota_{X_A}\omega=0, \forall A\in \mathfrak a\}.
\end{equation}
We recall that $d_s$ is the dimension of the stable foliation. On this bundle, we consider 
\begin{equation}
\label{eq:X}
\mathbf{X}:\mathfrak a\to \mathrm{Diff}(\mathcal M; \mathscr E_0^{d_s}), \ \mathbf{X}_A\omega:= \mathcal L_{X_A}\omega,
\end{equation}
where $\mathrm{Diff}(\mathcal M; \mathscr E_0^{d_s})$ denotes the space of differential operators acting on sections of $\mathscr E_0^{d_s}$.
This defines an admissible lift in the sense that it satisfies a Leibniz rule:
$$\forall f\in C^{\infty}(\mathcal M), \ \forall \omega \in \mathscr E_0^{d_s}, \quad \mathbf{X}_A(f\omega)=(X_A f)\omega+f\mathbf{X}_A \omega.$$
Thereafter, we will write $\mathcal D' (\mathcal M; \Lambda^{k}(E_u^*\oplus E_s^*))$ for the space of sections of currents of degree $d_s+d_u-k$ which are cancelled by the contraction $\iota_{X_A}$ for any $A\in \mathcal W$. They can be thought as linear combinations of elements of $\Lambda^{k}(E_u^*\oplus E_s^*)$ with distributional coefficients. We introduce a useful (Hölder continuous) splitting:
\begin{equation}
\label{eq:split2}
\Lambda^{q}(E_u^*\oplus E_s^*)=\bigoplus_{k=0}^{q}\big(\Lambda^kE_s^*\otimes \Lambda^{q-k}E_u^*\big)=:\bigoplus_{k=0}^{q} \Lambda_k^{q}.
\end{equation}

We now define {Ruelle-Taylor resonances} for the action on $d_s$-forms. They correspond to joint eigenvalues of the $X_A$ for $A\in \mathcal W$ on the space of distributions with wavefront set contained in $E_u^*$.
\begin{defi}
\label{cor}
We say that $\lambda \in \mathfrak a_{\mathbb C}^*$ is a Ruelle-Taylor resonance if and only if
\begin{equation}
\label{eq:vp}
\exists u\in \mathcal D'(M;\Lambda^{d_s}(E_s^*\oplus E_u^*))\setminus\{0\},\ \mathrm{WF}(u)\subset E_u^*,\ \forall A\in \mathcal W, \quad -\mathbf{X}_Au=\lambda(A) u.
\end{equation}
\end{defi}
We will write $\mathrm{Res}_{\mathbf{X}}^{d_s}$ for the set of Ruelle-Taylor resonances of the action on $d_s$-forms. The set $\mathrm{Res}_{\mathbf{X}}^{d_s}$ was shown to be discrete in \cite[Theorem 1]{GBGHW} and the corresponding spaces of joint (generalized) eigenfunctions are finite dimensional. Changing $-\mathbf{X}_A$ to $\mathbf{X}_A$ and replacing $E_u^*$ by $E_s^*$ in the wavefront set condition, we obtain the definition of a co-resonant state.
\begin{equation}
v\in \mathcal D'(M;\Lambda^{d_u}(E_s^*\oplus E_u^*))\setminus\{0\}, \mathrm{WF}(v)\subset E_s^*,\ \forall A\in \mathcal W, \ \ \mathbf{X}_Av=\lambda(A) v.
\end{equation}

The idea of \cite{GBGHW} (already present in \cite{BKL,BL,BT,GL,Fau08,Fau10} for the rank one case) was to study the joint spectral theory of the $(-\mathbf{X}_A)_{A\in \mathcal W}$ on specially designed functional spaces called \emph{anisotropic Sobolev spaces} $\mathcal H^{NG}$. Their precise construction, which we only sketch below, can be found in \cite[Section 4.1]{GBGHW}.

The anisotropic spaces $\mathcal H^{NG}$ are constructed using an \emph{anisotropic order function} $G$ which has to satisfy certain dynamical properties. The order function $G\in C^{\infty}(T^*\mathcal M)$ is homogeneous in the $\xi$ variable outside a compact set in $\xi$, it is negative in a conic neighborhood of $E_u^*$ and positive outside a larger conic neighborhood of $E_u^*$. Most importantly, it decreases along the trajectories of the symplectic lift $e^{tX_A^H}$ of $\varphi_t^A$ (see \cite[Definition 4.1]{GBGHW} for more precise statements). Using a quantization procedure $\mathrm{Op}$ (see \cite{Zw} for instance), we get elliptic pseudo-differential operators $\mathrm{Op}(e^{NG})$ (with $N\geq 0$) with variable order which can be inverted up to changing the operators to a lower order term. The anisotropic Sobolev spaces are defined to be $\mathcal H^{NG}:=\mathrm{Op}(e^{NG})^{-1}L^2(\mathcal M)$. Even though the construction of the Ruelle-Taylor resonances use these rather complicated functional spaces, they are only auxiliary tools and the different objects are independent of the particular choices made throughout the proof as seen in definition \eqref{eq:vp}.

In the rest of this subsection, we will prove that the leaf measure $m^s$ is a Ruelle-Taylor resonant state for the Ruelle-Taylor resonance $h_{\mathrm{top}}^{\mathcal W}$. Recall that using Theorem \ref{Theo1}, we can define $h_{\mathrm{top}}^{\mathcal W}\in \mathfrak a_\mathbb C^*$ such that $h_{\mathrm{top}}^{\mathcal W}(A)=h_{\mathrm{top}}(\varphi_1^A)$ for $A\in \mathcal W$ and extended by linearity in the rest of $\mathfrak a$.

\begin{prop}[Leaf measures are resonant states]
\label{muv}
The system of leaf measures $m^s$ $($resp. $m^u)$ from Proposition \ref{uniform} defines a section of $\mathcal D' (\mathcal M; \Lambda^{d_s}(E_u^*\oplus E_s^*))$ $($resp.  $\mathcal D' (\mathcal M; \Lambda^{d_u}(E_u^*\oplus E_s^*))$ which we will still denote by $m^s$ $($resp. $m^u)$. Moreover, one has
\begin{equation}
\label{eq:eee}
\begin{split}
&\mathrm{WF}(m^s)\subset E_u^*,\ \forall A\in \mathfrak a, \quad -\mathbf{X}_Am^s=h_{\mathrm{top}}^{\mathcal W} (A) m^s,
\\
&\mathrm{WF}(m^u)\subset E_s^*,\ \forall A\in \mathfrak a, \quad \mathbf{X}_Am^u=h_{\mathrm{top}}^{\mathcal W} (A) m^u.
\end{split}
\end{equation}
Hence $m^s$ $($resp. $m^u)$ is a resonant state $($resp. co-resonant state$)$ associated to the Ruelle-Taylor resonance $h_{\mathrm{top}}^{\mathcal W}$, which we will call the first resonance.

Moreover, for any $1\leq k\leq d_s$ and $\omega_k \in C^{0}(\mathcal M; \Lambda_k^{d_s})$, one has $m^u(\omega_k)=0$.
\end{prop}
\begin{proof}
We have to make sense of the pairing $m^u(\varphi)$ for $\varphi\in C^{\infty}(\mathcal M; \Lambda^{d_s}T^*\mathcal M)$. First, the compatibility statement on the different leaf measures $m^u_x$ allows us to only define the duality locally, so let us recall some facts on the \emph{local product structure} of the action. The Anosov decomposition \eqref{eq:split} integrates into $\mathcal W^{\bullet}$ for $\bullet=u,c,s$. Here, the central manifold $\mathcal W^c(x)$ corresponds to the orbit of $x$ under the Anosov action, that is 
$\mathcal W^c(x):=\{\varphi_1^A(x)\mid A \in \mathfrak a\}$. We can also define center-stable and center-unstable manifolds $\mathcal W^{cs}(x)$ and $\mathcal W^{cu}(x)$. Note that $\mathfrak a$ is equipped with its Lebesgue-Haar measure $dA$ which can be pushed forward to a $\kappa$-form on $\mathcal M$ using the injective Lie algebra homomorphism \eqref{eq:XLie}:
\begin{equation}
\label{eq:alpha1}
\alpha \in C^{\infty}(\mathcal M;\Lambda^\kappa T^*\mathcal M), \ \alpha:=X_*(dA),\quad \forall A\in \mathfrak a,\ \mathcal L_{X_A}\alpha=0.  
\end{equation}
 Since the Anosov decomposition is transverse and since $\mathcal M$ is compact, the local product structure (see for instance \cite[Chapter 4]{Pes}) assures that there exists $\delta_0>0$ small enough such that for any $\delta<\delta_0$ and $x\in \mathcal M$, 
$$\forall y\in \mathcal W^u(x,\delta),\ \forall z\in \mathcal W^s(x,\delta) ,\ \exists  A(z,y)\in \mathfrak a,  \quad \mathcal W^u(\varphi_1^{A(z,y)}z,\delta)\cap \mathcal W^s(y,\delta)\neq \emptyset .$$
Here, $\mathcal W^{\bullet}(x,\delta)$ denotes the ball of radius $\delta$ for the metric induced by the Riemannian metric. Moreover, if we require that $\|A(z,y)\|\leq \delta,$ then it is unique. The element $A(z,y)\in \mathfrak a$ is a multi-dimensional version of the Bowen time (defined in the classical rank one case, see \cite[Proposition 6.2.2]{FishHas} for instance). We can define the Bowen bracket  $[\varphi_1^{A(z,y)}z,y]$ to be the unique element of  $\mathcal W^u(\varphi_1^{A(z,y)}z,\delta)\cap \mathcal W^s(y,\delta)$.  We note that one has for any $ x,y\in \mathcal M$ which are {close} and $A_0\in \mathcal W$:
\begin{equation}
\label{eq:v}
A(\varphi_{t}^{A_0}y,\varphi_{t}^{A_0}x)=A(y,x),
\end{equation}
for any $t\geq 0$ such that the Bowen bracket is well defined.
The Bowen bracket map is defined to be
$$[\cdot,\cdot]: \mathcal W^{cu}(q,\delta)\times \mathcal W^s(q,\delta)\to \mathcal M,\ (x,y)\mapsto [x,y]. $$
A local rectangle $R_q$ centered at $q$ denotes an open neighborhood of $q$ obtained as the image of $\mathcal W^{cu}(q,\delta)\times \mathcal W^s(q,\delta)$ by the Bowen bracket map.
Let $\varphi\in C^{\infty}(\mathcal M; \Lambda^{d_s}T^*\mathcal M)$ be a smooth $d_s$-form supported in $R_q$. We can define the duality as follows:
\begin{equation}
\label{eq:muv}
m^u(\varphi):=\int_{\mathcal M}m^u\wedge \alpha \wedge \varphi=\int_{\mathcal W^u(q,\delta)}\left( \int_{R_q^{cs}(x)}e^{h_{\mathrm{top}}^{\mathcal W}(A(y,x))}( \alpha\wedge \varphi)(y)\right) dm^u(x),
\end{equation}
where $R_q^{cs}(x):= W^{cs}(x)\cap R_q.$
We see that the previous definition makes sense as $\varphi\wedge \alpha$ is a $(d_s+\kappa)-$form and as $h_{\mathrm{top}}^{\mathcal W}(A(y,x))$ is smooth in $y$ and continuous in $x$. Since $m^u_x$ is a measure on each local unstable leaf, the formula clearly defines a current $m^u\in \mathcal D' (\mathcal M; \Lambda^{d_u}(E_u^*\oplus E_s^*))$ of order $0$, that is, it can be tested against continuous sections.
Let $\omega_k \in C^{\alpha}(\mathcal M; \Lambda_k^{d_s})$ for some $k$. If $k\geq 1$, we can use \eqref{eq:muv} and the definition of the dual bundle $E_s^*$ \eqref{eq:dual} to get that $m^u(\omega_k)=0$.

To prove the first part of \eqref{eq:eee}, it suffices to consider a $d_s$-form $\varphi$ supported in $R_q$ and such that $e^{-t\mathbf X_{A_0}}\varphi$ is also supported in $R_q$ for some $A_0\in \mathcal W$. We have to prove that one then has $m^u(e^{t\mathbf{X}_{A_0}}\varphi)=e^{-th_{\mathrm{top}}^{\mathcal W}(A_0)}m^u(\varphi)$. We first use the Leibniz rule to get 
$$\mathbf{X}_{A_0}(\varphi \wedge \alpha)=(\mathbf{X}_{A_0}\varphi) \wedge \alpha+\varphi \wedge \mathbf{X}_{A_0}\alpha=(\mathbf{X}_{A_0}\varphi) \wedge \alpha, $$
where we used the invariance of $\alpha$, see \eqref{eq:alpha1}.
We use \eqref{eq:v} as well as \eqref{eq:phit} to get
 \begin{align*}
m^u(e^{t\mathbf{X}_{A_0}}\varphi)&=\int_{\mathcal W^u(q,\delta)}\left( \int_{R_q^{cs}(x)}e^{h_{\mathrm{top}}^{\mathcal W}(A(y,x))}\big(e^{t\mathbf{X}_{A_0}}\varphi \big)\wedge \alpha(y)\right) dm^u(x)
\\&=\int_{\mathcal W^u(q,\delta)}\left( \int_{R_q^{cs}(x)}e^{h_{\mathrm{top}}^{\mathcal W}(A(\varphi_{t}^{A_0}y,\varphi_{t}^{A_0}x))}e^{t\mathbf{X}_{A_0}}\big(\varphi \wedge \alpha\big)(y)\right) dm^u(x)
\\&=\int_{\mathcal W^u(q,\delta)}\left( \int_{R_q^{cs}(\varphi_{t}^{A_0} x )}e^{h_{\mathrm{top}}^{\mathcal W}(A(w,\varphi_{t}^{A_0}x))}(\varphi\wedge \alpha)(w)\right) dm^u(x),
\\
&=\int_{\mathcal W^u(q,\delta)}\left( \int_{R_q^{cs}(x)}e^{h_{\mathrm{top}}^{\mathcal W}(A(y,x))}(\varphi\wedge \alpha)(y)\right) d\big((\varphi_{t}^{A_0})_*m^u\big)(x)
\\&=e^{-th_{\mathrm{top}}^{\mathcal W}(A_0)}\int_{\mathcal W^u(q,\delta)}\left( \int_{R_q^{cs}(x)}e^{h_{\mathrm{top}}^{\mathcal W}(A(y,x))}(\varphi\wedge \alpha)(y)\right) dm^u(x)
\\&=e^{-th_{\mathrm{top}}^{\mathcal W}(A_0)}m^u(\varphi).
\end{align*}
For the wavefront set condition, we mimick the argument of \cite[Lemma 3.2]{Hum}. We consider a smooth $d_s$-form $\chi$ supported in $R_q$, a phase function $S\in C^{\infty}(\mathcal M)$ such that $dS(q)=\xi \notin E_s^*$ and compute
$$ m^u(e^{i\frac S h}\chi)=\int_{\mathcal W^u(q,\delta)}\left( \int_{R_q^{cs}(x)}e^{h_{\mathrm{top}}^{\mathcal W}(A(y,x))}e^{i\frac{S(y)} h}(\chi\wedge \alpha)(y)\right) dm^u(x).$$
Now, the proof is easier than for \cite[Lemma 3.2]{Hum} as the integrand is easily seen to be smooth along the weak-stable leaves uniformly in $x$ (in the sense explained in \cite[Lemma 3.2]{Hum}). This means that one can perform integration by parts in $y$ and show that the integrand is a $O(h^{\infty})$ and thus $m^u(e^{i\frac S h}\chi)=O(h^{\infty})$ as long as $dS$ does not vanish on $R_q^{cs}(x)$, which can be ensured near $q$ by the definition of $E_s^*$. This shows that $\xi \notin \mathrm{WF}(m^u)$ and thus $ \mathrm{WF}(m^u)\subset E_s^*$. In other words, $h_{\mathrm{top}}^{\mathcal W}$ is a Ruelle-Taylor resonance with the associated Ruelle co-resonant state given by $m^u$ using \eqref{eq:vp}.
\end{proof}

\subsection{Constructing the norm}
The arguments in \cite[Section 5]{GBGHW} show an alternative description of Ruelle-Taylor resonances as eigenvalues of a "joint propagator" operator $R$ which appears in the parametrix construction \cite[Lemma 4.14]{GBGHW}. This allows the authors of \cite{GBGHW} to give a precise description of resonant states on the \emph{critical axis} and is the starting point of the fine study of SRB measures they obtain in \cite{GBGW}. 

As already noticed by the author in \cite{Hum}, the validity of the results of \cite[Section 5]{GBGHW} only depends on the existence of a norm $\|.\|$ for which the propagator $e^{t\mathbf X_A}$ has sharp exponential growth. In the case of the action on functions (as considered in \cite{GBGW}), one could take the $L^{\infty}$-norm for which the propagator is bounded (hence giving a first resonance at $0$). When acting on other bundles of forms or when adding a smooth potential $V$, the norm can be constructed directly from the leaf measures $m^u$ (see \cite[Lemma 3.3, Proposition 4.4]{Hum}).
In this subsection, we follow this strategy and construct a norm $\|.\|_u$ on $\mathscr E_0^{d_s}$ satisfying
$$ \forall \omega \in \mathscr E_0^{d_s},\ \forall A\in \mathcal W, \quad \|e^{-\mathbf X_A}\omega\|_u\leq e^{h_{\mathrm{top}}(\varphi_1^A)}\|\omega\|_u.$$
This will allow us to prove that the critical axis of the action of $d_s$-forms is located exactly at $\mathcal C:=\{\lambda \in \mathfrak a_\mathbb C^* \mid \forall A\in \mathcal W, \ \mathrm{Re}(\lambda(A)) =  h_{\mathrm{top}}(\varphi_1^A)\}$ and that the Ruelle-Taylor resonances on the axis have no Jordan block. We follow closely the arguments of \cite[Lemma 4.3, Proposition 4.4]{Hum}. By Assumption \ref{assumption}, the bundle $\Lambda_0^{d_s}$ is trivial. By fixing a nowhere vanishing $d_s$-form, one can associate a density $|\varphi|$ to any $\varphi \in  C^{0}(\mathcal M; \Lambda_0^{d_s})$.
\begin{lemm}
\label{Crit}
We define a norm on $C^{0}(\mathcal M; \Lambda_0^{d_s})$ by
\begin{equation}
\label{eq:newnorm2}
\forall \varphi \in  C^{0}(\mathcal M; \Lambda_0^{d_s}), \quad \|\varphi\|_{u,0}:= m^u(|\varphi|).
\end{equation}
This norm satisfies the bound
\begin{equation}
\label{eq:bbound}
\forall  \varphi \in  C^{0}(\mathcal M; \Lambda_0^{d_s}), \ \forall A\in \mathcal W, \quad \|e^{-\mathbf{X}_A}\varphi\|_{u,0}\leq e^{h_{\mathrm{top}}^{\mathcal W}(A)}\|\varphi\|_{u,0}.
\end{equation}

\end{lemm}
\begin{proof}
Consider $\varphi \in  C^{0}(\mathcal M;\Lambda^{d_s}_0)$. The bundle is one dimensional and it thus makes sense to talk of $|\varphi|\wedge \alpha$ as a $d_s+\kappa$ density. If $\varphi(q)\neq 0$ then by continuity, $\varphi\neq 0$ on a small open set. Then $\|\varphi \|_{u,0}>0$ because $m^u$ gives a positive measure to any non-empty open set by \cite[Theorem A]{CarHer}. The bound \eqref{eq:bbound} follows from the fact that $m^u$ is a co-resonant state, see \eqref{eq:eee}.
\end{proof}
 We now use this norm and a "shift" to define inductively a norm on $C^0(\mathcal M; \Lambda_k^{d_s})$ for any $k$.
Consider the set of all finite covers by open sets which trivialize $E_s$ and $E_u$:
$$\mathcal C:=\{\mathcal U:=(U_j)_{1\leq j \leq n}\mid  \mathcal M=\cup_{j=1}^n U_j, \ U_j \text{ open and } E_s \text{ and } E_u \text{ are trivial on } U_j\}. $$
For any $\mathcal U\in \mathcal C$, let $\mathcal P(\mathcal U)$ be the set of partition of unity $(\chi_j)_{j=1}^n$ associated to the cover $\mathcal U$. Finally, we define the set of (normalized) local trivializations of $E_u$:
$$\mathscr V^u(\mathcal U):=\{ (X^j_{u,h})_{1\leq j \leq d_u}\in C^0(U_h; E_u) \mid (E_u)_{|U_h}=\mathrm{Span}\{X^j_{u,h}\}_{1\leq j \leq d_u}, \ \| X^j_{u,h}\|_{C^0}=1\},$$ 
and its dual conterpart
$$\mathscr F^u(\mathcal U):=\{ (Y^j_{u,h})_{1\leq j \leq d_s}\in C^0(U_h; E_u^*) \mid (E_u^*)_{|U_h}=\mathrm{Span}\{Y^j_{u,h}\}_{1\leq j \leq d_s}, \ \| Y^j_{u,h}\|_{C^0}=1\}.$$ 
\begin{prop}
\label{realnorm}
We define a norm inductively on $C^0(\mathcal M; \Lambda_{k}^{d_s})$ for $ k\geq 1$ by posing, for $f \in  C^{0}(\mathcal M; \Lambda_k^{d_s})$,
\begin{equation}
\label{eq:T}
 \|f\|_{u,k}:= \sup_{\mathcal U\in \mathcal C}\sup_{(\chi_j)\in \mathcal P(\mathcal U)}\max_{h=1}^n \sup_{(X^j_{u,h})\in \mathscr V^u}\sup_{(Y^i_{u,h})\in \mathscr F^u}\sum_{i=1}^{d_s}\sum_{j=1}^{d_u}\|\chi_h\iota_{X_{u,h}^j}f\wedge Y_{u,h}^{i}\|_{u,k-1}.
\end{equation}
This norm satisfies the bound
\begin{equation}
\label{eq:bbbound}
\forall  \varphi \in  C^{0}(\mathcal M; \Lambda_k^{d_s}),\ \forall A\in \mathcal W, \quad \|e^{- \mathbf{X}_A}\varphi\|_{u,k}\leq C(A) e^{h_{\mathrm{top}}^{\mathcal W}(A)-k\eta(A)}\|\varphi\|_{u,k}
\end{equation}
for some $C(A),\eta(A)>0$ that depend on the Anosov constants, see \eqref{eq:s}\footnote{In particular, they can be chosen uniformly in any proper subcone of $\mathcal W$.}.  
Consider $f\in C^{\infty}(\mathcal M; \mathscr E_{d_s})$ and consider its decomposition
$$f= \sum_{k=0}^{d_s} \omega_k, \quad \omega_k \in C^{\alpha}(\mathcal M; \Lambda_k^{d_s}). $$
We define a norm on $C^{\infty}(\mathcal M; \mathscr E_{d_s})$:
\begin{equation}
\label{eq:norm}
 \|f\|_{u}:=\sum_{k=0}^{d_s}\|\omega_k\|_{u,k}, \quad \forall A\in \mathcal W, \ \|e^{-t\mathbf{X}_A}f\|_u\leq Ce^{th_{\mathrm{top}}(\varphi_1^A)}\|f\|_u.
\end{equation}
\end{prop} 
The proof only relies on the Anosov property \eqref{eq:s} and the fact that $d_s$ is the dimension of the stable bundle. In particular, one can mimick the proof of \cite[Proposition 4.4]{Hum}, which is the equivalent construction for the rank one case. The fact that the central direction is of higher dimension is not a problem as our definition of $\mathscr E_0^{d_s}$ consists of forms which are cancelled by the contraction with any vector in the neutral direction.
\subsection{Critical axis}
We use the norm constructed in Proposition \ref{realnorm} to locate the critical axis. Again, we follow \cite[Section 5]{GBGHW} and \cite[Section 4]{Hum}. Recall from the parametrix construction \cite[Proposition 4.14]{GBGHW} that given a basis $A_1,A_2,\ldots A_\kappa, \in \mathcal W$, $\lambda \in \mathfrak a_{\mathbb C}^*$ and functions $\phi_j\in C_c^{\infty}(\mathbb R_+)$ such that $\int_0^{+\infty}\phi_j(t)dt=1$, one defines
\begin{equation}
\label{eq:R}
R(\lambda):=\prod_{j=1}^\kappa\int_{\mathbb R} e^{-t_j(\mathbf{X}_{A_j}-\lambda(A_j))}\phi_j(t)dt.
\end{equation}
The operator $R$ plays the role of a "joint" propagator and appears as a remainder in the parametrix construction of \cite[Propositions 4.14, 4.17]{GBGHW}.
 More precisely, for $\lambda\in \mathfrak a^*_{\mathbb C}$, there exists a parametrix $Q(\lambda)$ and  an anisotropic space $\mathcal H^{NG}$ on which
\begin{equation}
\label{eq:parametrix}
Q(\lambda)d_{\mathbf{X}+\lambda}+d_{\mathbf{X}+\lambda}Q(\lambda)=\mathrm{Id}-R(\lambda)\otimes \mathrm{Id}=:F(\lambda),
\end{equation}
where $d_{\mathbf{X}}$ denotes the Taylor differential associated to $\mathbf{X}$, see \cite[Section 3]{GBGHW} for a precise definition which we will not need.
Moreover, the operator $R(\lambda)$ acting on $\mathcal H^{NG}$ has essential spectral radius in $B(0,1/2)$ and  the spectrum outside $B(0,1/2)$ is discrete.  The crucial point for our purpose is that for any Ruelle-Taylor resonance $\lambda \in \mathrm{Res}_{X}^{d_s}$, one has $0\in \mathrm{Spec}(F(\lambda))$. We insist in the fact that the Ruelle-Taylor resonances do not depend on any choice done in the parametrix construction. We denote by $\Pi_0(\lambda)$ the spectral projector of $F(\lambda)$ on $0$. 
From the parametrix construction above, we deduce the position of the critical axis.
\begin{lemm}[Critical axis]
\label{lemm crit}
The Ruelle-Taylor resonances are located in 
$$\{\lambda \in \mathfrak a_\mathbb C^* \mid \forall A\in \mathcal W,  \ \mathrm{Re}(\lambda(A)) \leq h_{\mathrm{top}}(\varphi_1^A)\}.$$
\end{lemm}
\begin{proof}
Let $\lambda\notin \{\nu \in \mathfrak a_\mathbb C^* \mid \forall A\in \mathcal W, \ \mathrm{Re}(\nu(A)) \leq h_{\mathrm{top}}(\varphi_1^A)\}$ and choose a basis $(A_j)_{1\leq j \leq \kappa}\in\mathcal W^{\kappa}$ such that $\mathrm{Re}(\lambda(A_j)) > h_{\mathrm{top}}^{\mathcal W}(A_j)$ for all $j$. Then the spectral projector of $F(\lambda)=\mathrm{Id}- R(\lambda)$ is given by the integral
\begin{equation} 
\label{eq:proj}
\Pi_0(\lambda)=\frac{1}{2\pi i}\int_{|z|=\epsilon}(z\mathrm{Id}-R(\lambda))^{-1}dz,
\end{equation}
for a radius $\epsilon>0$ small enough. If $f\in C^{\infty}(\mathcal M;\mathscr E_0^{d_s})$, then using Proposition \ref{realnorm}, one has
\begin{align*}
\| R(\lambda)f\|_{u}&\leq \int_{(\mathbb R_+)^{\kappa}} \|e^{-\sum_{j=1}^\kappa t_j (\mathbf{X}_{A_j}-\lambda(A_j))}f\|_{u}\prod_{j=1}^\kappa \phi_j(t_j)dt_1\ldots dt_\kappa
\\ & \leq \int_{(\mathbb R_+)^{\kappa}} e^{-\sum_{j=1}^\kappa t_j (h_{\mathrm{top}}^{\mathcal W}(A_j)-\mathrm{Re}(\lambda(A_j)))}\|f\|_{u}\prod_{j=1}^\kappa \phi_j(t_j)dt_1\ldots dt_\kappa\leq \frac 1 2\|f\|_u
\end{align*}
if $\phi_j$ are chosen with support in $[T_j,+\infty[$ for $T_j$ large enough. In particular, this shows that $F(\lambda)$ is invertible and thus $\Pi_0(\lambda)=0$, meaning that $\lambda$ is not a Ruelle-Taylor resonance. 
\end{proof}
Having constructed the norm, we can now adapt the proof of \cite[Proposition 3.1]{Hum} and describe the resonant states on the critical axis using the joint propagator $ R$. The following proposition can be seen as an analogue of \cite[Lemma 5.2, Proposition 5.4]{GBGHW} for the action on $d_s$-forms. For a Ruelle-Taylor resonance $\lambda_0$, recall that we write $\mathrm{Res}_{\mathbf{X}}^{d_s}(\lambda_0)$ for the space of associated resonant states. We obtain the following important characterization of resonant states on the critical axis.
\begin{prop}
\label{Prop R}
Let $\lambda_0 \in\{\lambda \in \mathfrak a_\mathbb C^* \mid \mathrm{Re}(\lambda) = h_{\mathrm{top}}^{\mathcal W}\}$ be on the critical axis. Then $R(\lambda_0):\mathcal H^{NG}\to \mathcal H^{NG}$ has spectral radius equal to $1$ and $\lambda_0$ is a Ruelle-Taylor resonance if and only if $1$ is an eigenvalue of $R(\lambda_0)$. In this case, $1$ is the only eigenvalue of $R(\lambda_0)$ on the unit circle, the eigenvectors of $R(\lambda_0)$ coincide with the Ruelle-Taylor resonances at $\lambda_0$ and there is no Jordan block. We have the following convergence, as bounded operators of $\mathcal H^{NG}\to \mathcal H^{NG}$: 
\begin{equation}
\label{eq:convergence}
\Pi_0(\lambda_0)=\lim_{k\to +\infty}R(\lambda_0)^k,\quad \mathrm{Res}_{\mathbf{X}}^{d_s}(\lambda_0)=\Pi_0(\lambda_0)(\mathscr E_0^{d_s}).
\end{equation} 
More precisely, for $\omega \in \mathscr E_0^{d_s}$  and $\eta\in \mathscr E_0^{d_u}$, one has
\begin{equation}
\label{eq:conv}
\langle \Pi_0(\lambda_0)\omega, \eta\rangle=\lim_{k\to+\infty}\int_{(\mathbb R_+)^{\kappa}} e^{-\sum_{j=1}^\kappa t_j\lambda_j}\prod_{j=1}^\kappa \phi_j^{*k}(t_j)\langle e^{-\sum_j t_j \mathbf{X}_j}\omega,\eta\rangle_{\mathscr E_0^{d_s}\times \mathscr E_0^{d_u}} dt_1\ldots dt_\kappa,
\end{equation}
where we have denoted by $\lambda_j=\lambda(A_j)$, $X_j=X_{A_j}$ and $\phi_j^{*k}$ the $k$-th convolution product of $\phi_j$ with itself.
\end{prop}
\subsection{Resonances on the critical axis}
We show that the first resonance is simple and that the presence of resonances on the critical axis is linked to mixing properties of the action with respect to the measure of maximal entropy. We follow the strategy of \cite[Proposition 3.8]{Hum} which already borrowed important ideas from \cite{Cli}.
\begin{prop}[First resonance]
\label{First}
The first resonance $h_{\mathrm{top}}^{\mathcal W}$ is simple, in other words, the space of resonant states $($resp. co-resonant$)$ is one dimensional.
\begin{equation}
\label{eq:spann}
\begin{split}
&\{\eta \in \mathcal D'(\mathcal M;\Lambda^{d_s}(E_s^*\oplus E_u^*)),\  (-\mathbf{X}-h_{\mathrm{top}}^{\mathcal W})\eta=0, \ \mathrm{WF}(\eta)\subset E_u^*\}=\mathrm{Span}(m^s),
\\
&\{\theta \in \mathcal D'(\mathcal M;\Lambda^{d_u}(E_s^*\oplus E_u^*)),\  (\mathbf{X}-h_{\mathrm{top}}^{\mathcal W})\theta=0, \ \mathrm{WF}(\theta)\subset E_s^*\}=\mathrm{Span}(m^u).
\end{split}
\end{equation}
\end{prop}
\begin{proof}
We proceed in several steps, let us consider $\theta$ a co-resonant state associated to the first resonance $h_{\mathrm{top}}^{\mathcal W}$:
\begin{equation}
\label{eq:theta}
\theta \in \mathcal D'(\mathcal M;\Lambda^{d_u}(E_s^*\oplus E_u^*))\setminus\{0\},\quad  (\mathbf{X}-h_{\mathrm{top}}^{\mathcal W})\theta=0, \ \mathrm{WF}(\theta)\subset E_s^*.
\end{equation}

\textbf{The co-resonant state $\theta$ is of order $0$.}
The proof follows exactly the one from \cite[Proposition 4.5]{Hum} and relies only on the Anosov property \eqref{eq:s}, the description of co-resonant states given by \eqref{eq:conv} and the fact that the stable and unstable foliations are continuous. For our purpose, this means that $\theta$ can be tested against continuous sections.

\textbf{The restriction of $\theta$ on unstable manifolds is well defined.} By the wavefront set condition $\mathrm{WF}(\theta)\subset E_s^*$, one sees that the distributional product of $\theta$ and $[\mathcal W^u(x)]$ is well defined. Here,  $[\mathcal W^u(x)]$ denotes the integration current over the unstable manifold $\mathcal W^u(x)$ for a $x\in \mathcal M$. We define the restrictions of $\theta$ to be
\begin{equation}
\label{eq:thetax}
\theta_x(f):=(f[\mathcal W^u(x)], \theta)_{\mathcal H^{NG}\times \mathcal H^{-NG}}, \ \ x\in \mathcal M,\ f\in C^{\infty}(\mathcal M),
\end{equation}
where the bracket denotes the distributional pairing which coincides with the $\mathcal H^{NG}\times \mathcal H^{-NG}$ pairing for $N$ large enough by \cite[Lemma 2.11]{GBGW}. We see that $\theta_x$ is of order zero (as a product of such distributions) and $\theta_x$ identifies to a measure on $\mathcal W^u(x)$. We now prove that $\{\theta_x \mid x\in \mathcal M\}$ defines a system of leaf measure in the sense of Carrasco and Rodriguez-Hertz or Climenhaga et al.

\textbf{The system of measures $\{\theta_x \mid x\in \mathcal M\}$ satisfies a change of variable formula by the action.} This is a consequence of the eigenvalue equation satisfied by $\theta$. More precisely, for any smooth function $f\in C^{\infty}(\mathcal M)$ and $A\in \mathcal W$, 
 $$( f[\mathcal W^u(x)]), \theta)_{\mathcal H^{NG}\times \mathcal H^{-NG}}=e^{-th_{\mathrm{top}}^{\mathcal W}(A)}(e^{-t\mathbf{X}_A}( f[\mathcal W^u(x)]),\theta)_{\mathcal H^{NG}\times \mathcal H^{-NG}}.$$
Using $e^{-t\mathbf{X}_A}( f[\mathcal W^u(x)])=(e^{-tX_A} f)[\mathcal W^u(\varphi_t^A x)]$, then yields
\begin{equation}
\label{eq:varphit}
\theta_x(f)=e^{-th_{\mathrm{top}}^{\mathcal W}(A)}\theta_{\varphi_t^A x}(f(\varphi_{-t}^Ay))\iff (\varphi_{-t}^A)^*\theta_{\varphi_t^Ax}=e^{th_{\mathrm{top}}^{\mathcal W}(A)}\theta_x. 
\end{equation}

\textbf{Two co-resonant states with full support in each leaf are proportional.} We prove that if $\theta_1$ and $\theta_2$ are two co-resonant states such that $(\theta_{1})_x$ and $(\theta_2)_x$ have full support in each $\mathcal W^u(x)$, then they are proportional. Under the above assumption, we can apply \cite[Corollary 4.6]{CarHer} which states that the conditional measures of the measure of maximal entropy $m$ are equivalent to $(\theta_i)_x$. More precisely, if $\xi$ is a SLY partition, i.e it is subordinated to the partition by unstable manifolds, it is increasing $m$-a.e$(x)$ and $\xi(x)$ contains an open neighborhood of $x$ inside $\mathcal W^u(x)$ for any $x$, then the conditional measures $m^\xi_x$ satisfy the following:
\begin{equation}
\label{eq:Radon}
\frac{dm^{\xi}_x}{d(\theta_i)_x}=\frac{1}{(\theta_i)_x(\xi(x))},\quad i=1,2.
\end{equation}
Note that the fact that $\xi$ is SLY guarantees that the denominator does not vanish. Remark moreover that the density does not depend on $y\in \mathcal W^u(x)$ because we are studying the equilibrium state associated to the null-potential. In particular, we get
$$\frac{d(\theta_1)_x}{d(\theta_2)_x}=\frac{(\theta_1)_x(\xi(x))}{(\theta_1)_x(\xi(x))}:=g(x). $$
Let us prove that the density function $g$ is actually constant, we follow \cite[Corollary 3.12]{Cli}. First of all, notice that for any Borel set $Z\subset \xi(x)$, one has $g(x)=(\theta_1)_x(Z)/(\theta_2)_x(Z). $  This means that the Radon Nikodym derivative is well defined on the whole unstable manifold $\mathcal W^u(x)$, see also \cite[Corollary 4.6]{CarHer}. Fix a Borel set $X\subset \mathcal M$. We use the fact that $x\mapsto (\theta_i)_x(X\cap \mathcal W^u(x))$ is Hölder continuous (see \cite[Appendix]{Hum} for a more precise statement), this gives that $g$ is actually a continuous function. Moreover, it is invariant by $\varphi_t^A$ for any $A\in \mathcal W$. Indeed, using \eqref{eq:varphit}, 
$$g(\varphi_t^Ax)=\frac{(\theta_1)_{\varphi_t^Ax}(\varphi_t ^AZ)}{(\theta_2)_{\varphi_t^Ax}(\varphi_t ^AZ)}=\frac{\int_Ze^{th_{\mathrm{top}}^{\mathcal W}(A)}d(\theta_1)_x(y)}{\int_Ze^{th_{\mathrm{top}}^{\mathcal W}(A)}d(\theta_2)_x(y)}=g(x). $$
Thus, using the transitivity of the flow with respect to $m$ gives that $g$ is constant $m$-a.e but since $m$ has full support, $g$ is constant. 
We have shown that there exists $c>0$ such that for any $x\in \mathcal W^u(x)$, one has
$(\theta_1)_x=c(\theta_2)_x$. We would like to prove that $\theta_1=c\theta_2$, i.e that a co-resonant state can be reconstructed from its restrictions on unstable manifolds. This was done in \cite[Proposition 3.8, 4.5]{Hum} using a Fubini-like formula for measures with wavefront set in $E_s^*$. The proof is the same in this context and we obtain $\theta_1=c\theta_2$.

\textbf{There exists a basis of co-resonant states with full support in each leaf.}
By the previous discussion, it suffices to show that the (finite-dimensional) space of co-resonant states associated to the first resonance $h_{\mathrm{top}}^{\mathcal W}$ admits a basis $\theta_0,\ldots, \theta_n$ such that each $\theta_i$ has full support in each leave. For this, we use the description of co-resonant states of Proposition \ref{Prop R}.\footnote{The proposition is written for resonant states but we get an analogous result for co-resonant states by changing $X$ to $-X$.} We have proved that the space of co-resonant states $\mathrm{Res}_{\mathbf{X}}^{d_s}(h_{\mathrm{top}}^{\mathcal W})$ is equal to $\Pi_0^*(h_{\mathrm{top}}^{\mathcal W})(\mathscr E_0^{d_s})$\footnote{We denote by $\Pi_0^*(h_{\mathrm{top}}^{\mathcal W})$ the dual counterpart of $\Pi_0(h_{\mathrm{top}}^{\mathcal W})$.}. Actually, we see from the proof of \cite[Proposition 4.5]{Hum} that $\Pi_0^*(h_{\mathrm{top}}^{\mathcal W})$ extends to continuous sections. Moreover, using the Anosov property, the proof of \cite[Proposition 4.5]{Hum} implies that it is sufficient to consider continuous sections of $\Lambda^{d_s}E_u^*$, in other words
\begin{equation}
\label{eq:span}
\mathrm{Res}_{\mathbf{X}}^{d_s}(h_{\mathrm{top}}^{\mathcal W})=\Pi_0^*(h_{\mathrm{top}}^{\mathcal W})(C^0(\mathcal M; \Lambda^{d_s}E_u^*)).
\end{equation}
The advantage of doing this is that $C^0(\mathcal M; \Lambda^{d_s}E_u^*)$ is a line bundle which is trivial by Assumption \ref{assumption}. Let us consider a positive section $\omega_0\in C^0(\mathcal M; \Lambda^{d_s}E_u^*)$. Then for any $f\in \mathcal C^0(\mathcal M)$ and any continuous section $\eta\in C^0(\mathcal M; \Lambda^{d_u}E_s^*)$, one has
$$
\langle \Pi_0^*(h_{\mathrm{top}}^{\mathcal W})(f\omega_0), \eta\rangle=\lim_{k\to+\infty}\int_{(\mathbb R_+)^{\kappa}} e^{-\sum_{j=1}^\kappa t_jh_j}\prod_{j=1}^\kappa \phi_j^{*k}(t_j)\langle e^{\sum_j t_j X_j}(f\omega_0),\eta\rangle dt_1\ldots dt_\kappa,
 $$
 where $h_j:=h_{\mathrm{top}}^{\mathcal W}(A_j)$. It is then clear that, one has
 $$ |\langle \Pi_0^*(h_{\mathrm{top}}^{\mathcal W})(f\omega_0), \eta\rangle|\leq \|f\|_0\langle \Pi_0(h_{\mathrm{top}}^{\mathcal W})^*(\omega_0), |\eta|\rangle.$$ Denote by $\theta_0:= \Pi_0(h_{\mathrm{top}}^{\mathcal W})^*(\omega_0)$, we have showed that
 \begin{equation}
 \label{eq:Radon2}
 \forall \theta \in  \mathrm{Res}_{\mathbf{X}}^{d_u,*}(h_{\mathrm{top}}^{\mathcal W}),\ \exists C(\theta)>0,\ \forall \eta \in C^0(\mathcal M; \Lambda^{d_s}E_s^*), \quad |\langle \theta, \eta \rangle|\leq C\langle \theta_0, |\eta|\rangle.
 \end{equation}
We see that if $\theta_0$ did not have full support in each leaf, then it would be the case of all co-resonant states. However, we know from Proposition \ref{muv} that $m^u$ is a co-resonant state and it has full support in each leaf by \cite[Theorem A]{CarHer}. In particular, we see that for any $\theta \in \mathrm{Res}_{\mathbf{X}}^{d_s}(h_{\mathrm{top}}^{\mathcal W})$, choosing $K>C(\theta)>0$ we have that $\theta+K\theta_0$ has full support in each leaf. If $\theta_0,\ldots, \theta_n$ was a basis, then it is still the case for $\theta_0,\theta_1+K_1\theta_0,\ldots, \theta_n+K_n\theta_0$ and this provides the desired basis. This concludes the proof of the proposition as the previous step shows that all of these co-resonant states are proportional.
\end{proof}
We finish by relating the presence of extra resonances on the critical axis to mixing properties of the measure $m$ and we follow the structure of the argument of \cite[Proposition 4.6]{Hum}. The weak-mixing property of $m$ was obtained by Carrasco and Rodriguez-Hertz in \cite[Theorem C]{CarHer}. Actually, the two authors proved the stronger Bernoulli property but we will not need to use it here.
\begin{prop}[No other resonance on the critical axis]
\label{label mixing}
Under Assumption \ref{assumption}, there are no other resonances on $\mathcal C=\{\lambda\in \mathfrak a^*\mid \mathrm{Re}(\lambda)=h_{\mathrm{top}}^{\mathcal W}\}.$
\end{prop}
\begin{proof}
Consider for a contradiction a resonance $\mu \in \mathcal C$ and an associated co-resonant state $\theta$, recall from the proof of Proposition \ref{First},  $\theta$ is of order $0$ and that  there is a $C>0$ such that 
$$\forall  \omega\in C^{\infty}(\mathcal M; \Lambda^{d_s}(E_u^*\oplus E_s^*)), \quad \langle\omega, \theta\rangle \leq C|\langle \omega, m^u\rangle|, $$
where we recall that $\langle\omega, \theta\rangle=\int_{\mathcal M}\omega \wedge \alpha \wedge \theta$ and the pairing is meant distributionally.
Using an approximation argument, this gives:
\begin{equation}
\label{eq:RadonNykodym}
\forall f\in C^{\infty}(\mathcal M),\ |\langle f m^s,\theta\rangle|\leq C\langle |f| m^s,m^u\rangle=Cm(|f|).
\end{equation}
 In other words, we have $m^s\wedge \theta \wedge \alpha\ll m$ (where $m$ is the unique measure of maximal entropy constructed in Theorem \ref{Theo1}) with bounded density $h\in L^{\infty}(\mathcal M, m)$. 

We use \cite[Theorem VII.14]{ReSi} to see that the flow $\varphi_1^{A}$ for $A\in \mathcal W$ is weakly mixing with respect to $m$ if and only if the only $L^2$ eigenvalue of $X_A$ is $1$ and it is a simple eigenvalue. In other words, if the system
\begin{equation}
\label{eq:thrid}
\begin{cases}
X_Af=i\lambda f
\\ f\in L^2(\mathcal M,m)
\end{cases}
\end{equation}
has no nontrivial solution except for $\lambda=0$ and $f$ constant. Since $\theta$ is a co-resonant state for $\mu\in \mathcal C$, we see that the density $h$ satisfies $X_Ah=\mathrm{Im}(\mu(A))h$. Moreover, since $L^{\infty}(\mathcal M,m)\subset L^2(\mathcal M,m)$ by compactness of $\mathcal M$, this means that $h$  is a solution \eqref{eq:thrid} and this implies that $h$ is constant as well as $\lambda=0$. But this implies that $\theta$ is a co-resonant state at the first resonance which contradicts Proposition \ref{First}.
\end{proof}
\section{Bowen-type formula}
\label{Bowen type formula}
In this section, we prove Theorem \ref{theoBowen} and Corollary \ref{corr}.
This section will follow closely the arguments of \cite[Section 4]{GBGW}. We will focus on the parts of the proof that need adaptation to our setting and refer to \cite{GBGW} for the details. 

Recall that a point $x\in \mathcal M$ is called periodic if there is a $A\in \mathfrak a$ such that $\tau(A)x=x$. If $A\in \mathcal W$, it is known that $T_x:=\{\tau(A')x\mid A'\in \mathfrak a\}\subset \mathcal M$ is a $\kappa$-dimensional torus (see \cite[Lemma 4.1]{GBGW}). The set of all periodic torii is denoted by $\mathcal T$ and for $T\in \mathcal T$, the associated lattice is denoted by  $L(T):=\{A'\in \mathfrak a\mid \tau(A')x=x\}\subset \mathfrak a$. The map $\varphi_1^A$ is transversally hyperbolic to $T$ and we define the Poincaré map to be $\mathcal P_A(x):=d_x(\varphi_{-1}^A)|_{E_u(x)\oplus E_s(x)}$. As a consequence, $\mathrm{det}(\mathrm{Id}-\mathcal P_A)$ does not vanish and its value on $T$ does not depend on which $x\in T$ we choose to compute it.

The invariant torus $T_x:=\tau(\mathbb A (x))\cong \mathbb A/L(T)$ is equipped with a natural measure obtained by pushing-forward the Haar measure. This measure will be denoted by $\lambda_T$ and it thus makes sense to integrate a function $f\in C^0(\mathcal M)$ on the torus $T_x$. 

The argument of \cite{GBGW} (already present in \cite{DyaZw} for the rank one case) consists in taking the \emph{flat trace} of a shifted resolvent of the flow. Then, using Guillemin's trace formula, it can be expressed using the periodic orbits of the Anosov action. In our setting, no natural notion of "joint resolvent" exists so applying the strategy above is not immediate. However, as already noticed in \cite{GBGW}, the joint resolvent can be replaced by a "joint propagator" $R$ defined in \eqref{eq:R}. Recall that the joint propagator $R$ was already used crucially in the proof of Proposition \ref{First}, which established the fine study of resonances on the critical axis and is closely related to the resonant states at the first resonance $h_{\mathrm{top}}^{\mathcal W}$. Recall also that $\mathcal W$ is the positive Weyl chamber of a transversely hyperbolic element $A_0\in \mathfrak a$.
\begin{defi}
Let $\psi\in C_c^{\infty}(\mathcal W)$ be such that $\int_{\mathcal W}\psi(A)dA=1$. For any  $0\leq m \leq n-\kappa$, $\lambda\in \mathfrak a^*$, $f\in C_c^{\infty}(\mathcal M)$ and $s\in \mathbb  C$, we define
\begin{equation}
\label{eq:Laplace}
R_{\psi,m}(\lambda):=\int_{\mathcal W}e^{-\mathbf{X}_A-\lambda(A)}|_{\mathscr E_0^m}\psi(A)dA,
\end{equation}
as well as
\begin{equation}
\label{eq:Shift}
T^\lambda_{\psi,f,m}(s):=fR_{\psi,m}(\lambda)(R_{\psi,m}(\lambda)-s)^{-1}.
\end{equation}
\end{defi}
The operator $R_{\psi,m}$ is a joint propagator\footnote{We already defined a joint propagator of this form in \eqref{eq:R} but we recall the definition here since it is slightly more general.} of the action and $T^\lambda_{\psi,f,m}$ will play the role of the shifted resolvent.

\textbf{Step 1: Guillemin trace formula.}
The first step consists in relating a suitable notion of trace of $R_{\psi,m}$ to the periodic orbits of the action. The {flat trace} is an extension of the usual trace to a subclass of distributions obeying some wave-front set condition, see \cite[Theorem 8.2.4]{Hor} for a precise statement. The wavefront set condition can be checked on $R_{\psi,m}$ using source and sink estimates and this is done in the proof of \cite[Proposition 4.2]{GBGW}. As already noticed there, their argument extends directly to smooth vector bundles and we will apply it to the bundles $\mathscr E_0^m$. 
\begin{prop}[Trace of the shifted resolvent]
\label{Guillemin}
Under Assumption \ref{assumption} and with the notations of the previous definition, the flat trace of the shifted resolvent 
\begin{equation}
\label{eq:trb}
Z_{\psi,f,m}(s,\lambda):=\tr^{\flat}(T^\lambda_{\psi,f,m}(s)),
\end{equation}
is well defined for $\lambda \in \mathfrak a_{\mathbb C}^*$ with $\mathrm{Re}(\lambda)$ large enough and $s\in B_{\mathbb C}(1,1/2)$. Moreover, it admits a meromorphic extension to $ B_{\mathbb C}(1,1/2)\times \mathfrak a_{\mathbb C}^*$ with the following expansion :
\begin{equation}
\label{eq:expan}
Z_{\psi,f,m}(s,\lambda)=\sum_{k=1}^{+\infty}s^{-k}\sum_{T\in \mathcal T}\sum_{A\in \mathcal W\cap L(T)}\left(\int_T fd\lambda_T\right)\frac{\tr(\Lambda^m \mathcal P_A)e^{-\lambda(A)}\psi^{*k}(A)}{|\mathrm{det}(\mathrm{Id}-\mathcal P_A)|}.
\end{equation}
Here, $\psi^{*k}$ denotes the $k$-th convolution product of $\psi$. Finally, if one replaces $\psi$ with $\psi_\sigma:=\psi(\cdot{-}\sigma)$, then $Z_{\psi,f,m}$ depends continuously\footnote{The topology here is given by uniform convergence on compact subsets of the holomorphic region in $\mathfrak a_{\mathbb C}^*\times B_{\mathbb C}(1,1/2)$.}  on $\sigma$ in a small neighborhood of $0$.
\end{prop}
\begin{proof}
This can be seen as the combination of \cite[Propositions 4.4 and 4.6]{GBGW}. The first proposition, which is Guillemin's trace formula, extends without further work to our setting as already noticed in the paper and reads as follows. Let $\mathcal C \subset \mathcal W$ be a closed cone, then there is $C>0$ such that for any $h\in C^0(\mathcal M\times \mathcal W)$ with support in $\mathcal M\times \mathcal C$ such that $\mathrm{sup}_{x\in \mathcal M,A\in \mathcal C}e^{C|A|}|h(x,A)|<+\infty$ and for any $0\leq m \leq n- \kappa$, one has (see \cite[Equation (4.2)]{GBGW})
\begin{equation}
\label{eq:Guillemin}
\tr^{\flat}\left(\int_{\mathcal W}he^{-\mathbf{X}_A}|_{\mathscr E_0^m}dA\right)=\sum_{T\in \mathcal T}\sum_{A\in \mathcal W\cap L(T)}\frac{\tr(\Lambda^m \mathcal P_A)\int_T h(x,A)d\lambda_T(x)}{|\mathrm{det}(\mathrm{Id}-\mathcal P_A)|}.
\end{equation}
In particular, the proof of \cite[Proposition 4.6]{GBGW} gives for  $\psi\in C_c^{\infty}(\mathcal W)$ with $\int_{\mathcal W}\psi(A)dA=1$ and support contained in a small conic neighborhood $\mathcal C$ of $A_0\in \mathcal W$, for any $f\in C^{\infty}(\mathcal M)$ and any $k\geq 0$,
\begin{equation}
\label{eq:Guillemin2}
\tr^{\flat}\left(fR_{\psi,m}(\lambda)^k\right)=\sum_{T\in \mathcal T}\sum_{A\in \mathcal W\cap L(T)}\left(\int_T fd\lambda_T\right)\frac{ \tr(\Lambda^m \mathcal P_A)e^{-\lambda(A)}\psi^{*k}(A)}{|\mathrm{det}(\mathrm{Id}-\mathcal P_A)|}.
\end{equation}
Recall from Proposition \ref{Prop R} that the joint propagator is bounded on the anisotropic space $\mathcal H^{NG}$. This means that one has the following convergence, in $\mathcal L(\mathcal H^{NG})$:
$$ \forall |s|\gg 1, \quad T^\lambda_{\psi,f,m}(s)=\sum_{k=1}^{+\infty}s^{-k}fR_{\psi,m}(\lambda)^k$$
In particular, we see that the expansion \eqref{eq:expan} follows (at least formally) from applying the flat trace to both sides and using \eqref{eq:Guillemin2}. The equality is proven rigorously using an approximation argument, see \cite[Lemma 4.9]{GBGW}.
\end{proof}
\textbf{Step 2: using the structure on the critical axis.}
The second step consists in relating the left hand side of \eqref{eq:expan} with the measure of maximal entropy $m$. This is done by noticing that the residue at a pole $\lambda_0$ of $Z_{\psi,1,m}$ corresponds to the spectral projector of $R_{\psi,m}$ for the eigenvalue $\lambda_0$. However, by Proposition \ref{Prop R}, this means that $\lambda_0$ is a Ruelle-Taylor resonance and the spectral projector coincides with the projector on the space of resonant states. We will need the following notation. Given $\psi\in C_c^{\infty}(\mathcal W)$ and $\lambda\in \mathfrak a_{\mathbb C}^*$, we define the Laplace transform to be
$$\hat{\psi}(\lambda):=\int_{\mathfrak a}e^{-\lambda(A)}\psi(A)dA. $$
We prove the following result, which is an adaptation of \cite[Proposition 3.10]{GBGW}.
\begin{prop}
\label{prop2}
Let $\psi\in C_c^{\infty}(\mathcal W;\mathbb R_+)$ such that $\int_{\mathcal W}\psi(A)dA=1$. Then there is $\epsilon>0$ such that for any $k\geq 0$
\begin{equation}
\label{eq:O}
\sum_{T\in \mathcal T}\sum_{A\in \mathcal W\cap L(T)}\psi^{*k}(A)e^{-h_{\mathrm{top}}^{\mathcal W}(A)}\frac{\tr(\Lambda^{d_s}\mathcal P_A)\int_T fd\lambda_T}{|\mathrm{det}(\mathrm{Id}-\mathcal P_A)|}=m(f)+O(e^{-\epsilon k}).
\end{equation}
Moreover, if $\psi=\psi_\sigma$, then the remainder is uniform locally in $\sigma$.
\end{prop}
\begin{proof}
We follow the proof of \cite[Proposition 4.10]{GBGW}. Using \eqref{eq:expan}, the task reduces to estimating the coefficients $c_k$ in the expansion of 
$$Z_{\psi,f,d_s}(s,h_{\mathrm{top}}^{\mathcal W})=\sum_{k\geq 0}c_ks^{-k}.$$
By Cauchy's formula, if we prove that this meromorphic function has a pole at $s_0=1$ with a residue equal to $m(f)$ and no other poles outside of $B(0,1-\epsilon')$ for some $\epsilon'>0$, then one gets 
$$c_k=m(f)+\frac 1{2i\pi}\int_{\partial B(0,1-\epsilon')}Z_{\psi,f,d_s}(s,h_{\mathrm{top}}^{\mathcal W})s^{k-1}ds=m(f)+O(e^{-\epsilon k}) $$
for some $\epsilon>0$. Note also that the remainder would be uniform is $\sigma$ using Proposition \ref{Guillemin} which shows that $Z_{\psi_{\sigma},f,d_s}$ depends continuously in $\sigma$ in a neighborhood of $0$. 

We first see that the poles of $Z_{\psi,f,d_s}(\cdot,h_{\mathrm{top}}^{\mathcal W})$ are poles of $(s-R_{\psi,d_s}(h_{\mathrm{top}}^{\mathcal W}))^{-1}$. Moreover, near a pole $s_0$, we get the expansion
$$(s-R_{\psi,d_s}(h_{\mathrm{top}}^{\mathcal W}))^{-1}=\sum_{j\geq 0}\frac{(R_{\psi,d_s}(h_{\mathrm{top}}^{\mathcal W})-s_0)^j\Pi(h_{\mathrm{top}}^{\mathcal W},s_0)}{(s-s_0)^j}+h(s), $$
where $h(s)$ is the holomorphic part in $s$, the summation is finite and $\Pi(h_{\mathrm{top}}^{\mathcal W},s_0)$ is the spectral projector of $R_{\psi,d_s}(h_{\mathrm{top}}^{\mathcal W})$ on the eigenspace associated to $s_0$. Since the joint propagator $R_{\psi,d_s}(h_{\mathrm{top}}^{\mathcal W})$ is Fredholm on the anisotropic spaces, the characteristic space $E(s_0)$ associated to $s_0$ is finite-dimensional and it can be further split into joint-eigenspaces of the action as $R_{\psi,d_s}(h_{\mathrm{top}}^{\mathcal W})$ commutes with the action. We thus choose $u\in E(s_0)$ such that $-\mathbf{X}_Au=\lambda_0(A)u$ for all $A\in \mathcal W$. Since we must have $u\in \mathcal H^{NG}$ for some suitable choice of $N,G$, this first implies that $\lambda_0$ is a Ruelle-Taylor resonance. Next, we get
$$R_{\psi,d_s}(h_{\mathrm{top}}^{\mathcal W})u=\int_{\mathcal W}e^{-\mathbf{X}_A-h_{\mathrm{top}}^{\mathcal W}(A)}u\psi(A)dA=\hat \psi(h_{\mathrm{top}}^{\mathcal W}-\lambda_0)u. $$
We can now use Lemma \ref{lemm crit} which gives $h_{\mathrm{top}}^{\mathcal W}-\mathrm{Re}(\lambda_0)\geq 0$ on $\mathcal W$. In particular, one has
$$|\hat \psi(h_{\mathrm{top}}^{\mathcal W}-\lambda_0)|\leq \int_{\mathcal W}e^{\mathrm{Re}(\lambda_0-h_{\mathrm{top}}^{\mathcal W})(A)}\psi(A)dA\leq \int_{\mathcal W}\psi(A)dA=1. $$
We have equality if and only if $e^{(\mathrm{Re}(\lambda_0)-\lambda_0)(A)}=1$ and $e^{(\mathrm{Re}(\lambda_0)-h_{\mathrm{top}}^{\mathcal W})(A)}=1$ on the support of $\psi$, which can only occur if $\lambda_0=h_{\mathrm{top}}^{\mathcal W}$. Since the spectral radius of $R_{\psi,d_s}(h_{\mathrm{top}}^{\mathcal W})$ is equal to $1$ and the spectrum of $R_{\psi,d_s}(h_{\mathrm{top}}^{\mathcal W})$ is discrete outside $B(0,1/2)$ by Proposition \ref{Prop R},  this proves that $s_0=1$ is a leading pole.

Again thanks to Proposition \ref{Prop R}, there is no Jordan block at $s_0=1$ and the spectral projector $\Pi(h_{\mathrm{top}}^{\mathcal W}, s_0)$ is equal to the projector onto the space of resonant states at the first Ruelle-Taylor resonance. We can now use Proposition \ref{First} to compute the residue. Indeed, the spectral projector at the first resonance writes 
\begin{equation}
\label{eq:projj}
\Pi_0(h_{\mathrm{top}}^{\mathcal W})=m^u(\cdot)m^s,
\end{equation}
where the action of $m^u$ is defined in \eqref{eq:muv} and the systems of leaf measures $m^u$ and $m^s$ are normalized so that $m=m^u\wedge \alpha \wedge m^s$ is a probability measure.
In particular, 
\begin{align*}\mathrm{Res}(Z_{\psi,f,d_s}(\cdot, h_{\mathrm{top}}^{\mathcal W}), 1)&=\tr(fR_{\psi,d_s}(h_{\mathrm{top}}^{\mathcal W})\Pi_0(h_{\mathrm{top}}^{\mathcal W}))=\tr(f\Pi_0(h_{\mathrm{top}}^{\mathcal W}))
\\&=m^u(f\times m^s)=m(f).
\end{align*}
This concludes the proof of the proposition.
\end{proof}
\textbf{Step 3: expressing $m$ in terms of periodic orbits.}
We prove the following.
\begin{lemm}
With the notations of Theorem \ref{theoBowen}, one has for any $f\in C^{\infty}(M)$,
\begin{equation}
\label{eq:deduce}
m(f)=\lim_{N\to +\infty}\frac{1}{|\mathcal C_{aN,bN}|}\sum_{T\in \mathcal T}\sum_{A\in \mathcal C_{aN,bN}\cap L(T)}e^{-h_{\mathrm{top}}(\varphi_1^A)}\frac{\tr(\Lambda^{d_s}\mathcal P_A)\int_T fd\lambda_T}{|\mathrm{det}(\mathrm{Id}-\mathcal P_A)|}.
\end{equation}
\end{lemm}
\begin{proof}
We assume without loss of generality that $f\geq 0$. Define a measure $\nu$ on $\mathcal W$ by
$$ \nu= \sum_{T\in \mathcal T}\sum_{A\in \mathcal W\cap L(T)}e^{-h_{\mathrm{top}}(\varphi_1^A)}\frac{\tr(\Lambda^{d_s}\mathcal P_A)\int_T fd\lambda_T}{|\mathrm{det}(\mathrm{Id}-\mathcal P_A)|}\delta_A.$$
We follow the argument of \cite[Proof of Theorem 4]{GBGW} and choose a basis $(A_j)_{1\leq j\leq \kappa}$ of $\mathfrak a$ such that $A_1\in \mathcal W$ and $A_j\in \mathrm{ker}(e_1)$ for $j\geq 2$, where $e_1\in \mathfrak a^*$ satisfies $e_1(A_1)=1$. Let $\Sigma:=\mathcal C\cap\{A_1+\sum_{j=2}^\kappa t_jA_j\mid t_j\in \mathbb R\}$. We consider 
\begin{itemize}
\item an even non-negative function $\psi\in C_c^{\infty}(-r/2,r/2)$ for $r$ smaller than the distance of $\mathcal C$ to the boundary of the Weyl chamber and such that $\int_{\mathbb R}\psi=1$. For any $\sigma\in \mathbb R^{\kappa}$, we will write $\psi_\sigma(t):=\prod_{i=1}^\kappa \psi(t_i-\sigma_i)$.
\item A function $q\in C_c^{\infty}(\Sigma, \mathbb R_+)$ with small support. We write $Q:=\int_{{\mathbb R}^{\kappa-1}}q(\bar t)d\bar t$.
\item A function $\omega \in C_c^{\infty}((0,1);[0,1])$ and we write $W:=\int_0^1w$.
\end{itemize}
For $\theta\in \mathbb R^{\kappa-1}$, we define $\sigma(\theta):=(1,\theta)\in \Sigma$ and consider for an integer $N$, 
\begin{equation}
\label{eq:FN}
F_N(t):=\frac 1 N\sum_{k=1}^N\int_{{\mathbb R}^{\kappa-1}}\omega\big(\frac k N\big)\psi_{\sigma(\theta)}^{*k}(t)q(\theta)d\theta.
\end{equation}
Using \eqref{eq:O} and the uniformity of the estimate in $\sigma$, we compute $\lim_{N\to+\infty}\nu(F_N)$:
\begin{align*}
&\lim_{N\to+\infty}\sum_{T\in \mathcal T}\sum_{A\in \mathcal W\cap L(T)}e^{-h_{\mathrm{top}}^{\mathcal W}(A)}\frac{\tr(\Lambda^{d_s}\mathcal P_A)\int_T fd\lambda_T}{|\mathrm{det}(\mathrm{Id}-\mathcal P_A)|}F_N(A)
\\&=\lim_{N\to+\infty}\frac 1 N \sum_{k=1}^N\int_{{\mathbb R}^{\kappa-1}}\sum_{T\in \mathcal T}\sum_{A\in \mathcal W\cap L(T)}\omega\big(\frac k N\big)\frac{\tr(\Lambda^{d_s}\mathcal P_A)\int_T fd\lambda_T}{|\mathrm{det}(\mathrm{Id}-\mathcal P_A)|}\psi_{\sigma(\theta)}^{*k}(A)e^{-h_{\mathrm{top}}^{\mathcal W}(A)}q(\theta)d\theta
\\&=\lim_{N\to+\infty}\frac 1 N \sum_{k=1}^N\omega\big(\frac k N\big)\int_{{\mathbb R}^{\kappa-1}}\big(m(f)+O(e^{-\epsilon k})\big)q(\theta)d\theta=WQm(f).
\end{align*}
We have thus proven that
\begin{equation}
\label{eq:nuFN}
\lim_{N\to+\infty}\nu(F_N)=WQm(f).
\end{equation}
Following the argument of \cite[Proof of Theorem 4]{GBGW}, we start by defining $h(t):=t_1^{1-\kappa}\omega(t_1)q(\bar{t}/t_1)$, where $t=(t_1,\bar t)$. We will need \cite[Equation (4.25)]{GBGW}:
\begin{equation}
\label{eq:GBGW}
\|F_N(t)-N^{-\kappa}h(t/N)\|_{C^0}=o(N^{-\kappa}).
\end{equation}
If we choose $q\equiv 1$ on an open set $U\subset \Sigma$ and $\omega(t)=t_1^{\kappa-1}$ on $(\epsilon, 1-\epsilon)$, for $N$ large enough,
$$F_N(t)\geq \frac{1}{2N^{\kappa}}\mathbf{1}_{[\epsilon N,(1-\epsilon)N]}(t_1)\mathbf{1}_U(\bar t/t_1)=\frac{1}{2N^{\kappa}}\mathbf{1}_{ \mathcal C_{\epsilon N,(1-\epsilon)N}(U)}, $$
where $\mathcal C_{a,b}(U):=\{t\mid t_1\in [a,b], \ \bar t/t_1\in U\}$. Using \eqref{eq:nuFN} yields
$$\nu( \mathcal C_{\epsilon N,(1-\epsilon)N}(U))\leq 2N^{\kappa}\nu(F_N(t))\leq 3N^{\kappa}QWm(f). $$
From this, we deduce that $\nu(\mathcal C_{0,N}(U))=O(N^{\kappa})$ by letting $\epsilon \to 0$. Using a finite cover by small open sets, we also deduce that $\nu(\mathcal C\cap \{\|A\|\leq N\})=O(N^{\kappa})$.
We now consider general functions $q,\omega$ and remark that there always exists an open set $U'\subset \Sigma$ containing the support of $F_N$ and $h(\cdot/N)$. In particular, using \eqref{eq:GBGW} gives
$$\lim_{N\to+\infty}N^{-\kappa}\nu(h(\cdot/N))=\lim_{N\to+\infty}\big(\nu(F_N)+o(N^{-\kappa})\nu(U')\big)=\lim_{N\to+\infty}\nu(F_N)=WQm(f). $$
To conclude, it suffices to approximate $\mathbf{1}_{ \mathcal C_{\epsilon N,(1-\epsilon)N}(U)}$ by smooth functions. Consider
$$q_j\in C^{\infty}_c(\Sigma), \ j=1,2, \ q_1\leq \mathbf{1}_U\leq q_2, \ \int q_j= |U|+O(\epsilon), $$
as well as
$$\omega_j\in C^{\infty}((0,1);[0,1]), \ j=1,2, \ \omega_1\leq t_1^\kappa \mathbf{1}_{[a,b]}\leq \omega _2, \ \int \omega_j= \int_a^bt_1^\kappa+O(\epsilon). $$
Write $h_j(t):=t_1^{1-\kappa}\omega_j(t_1)q_j(\bar t/t_1)$. We obtain
$$N^{-\kappa}\nu(h_1(\cdot/N))\leq N^{-\kappa}\nu(\mathcal C_{aN,bN}(U))\leq  N^{-\kappa}\nu(h_2(\cdot/N)).$$
Writing $V_{a,b}=\int_a^bt_1^{\kappa-1}dt_1$, we obtain
\begin{align*}
(V_{a,b}-\epsilon)(|U|-\epsilon)m(f)&\leq \liminf_{N\to +\infty}N^{-\kappa}\nu(\mathcal C_{aN,bN}(U))\leq \limsup_{N\to +\infty}N^{-\kappa}\nu(\mathcal C_{aN,bN}(U))
\\&\leq (V_{a,b}+\epsilon)(|U|+\epsilon)m(f).
\end{align*}
Finally, using $|\mathcal C_{aN,bN}(U)|=N^{-\kappa}|U|\times V_{a,b}$, we deduce \eqref{eq:deduce} by letting $\epsilon \to 0$.
\end{proof}
\textbf{Step 4: getting rid of the Poincaré factor.}
The last step consists in deducing \eqref{eq:bowen} from \eqref{eq:deduce}. For this, we first use the orientability assumption of the stable manifold to obtain 
$$ \forall T\in \mathcal T, \ \forall A\in \mathcal W\cap L(T), \quad |\mathrm{det}(\mathrm{Id}-\mathcal P_A)|=(-1)^{d_s}\mathrm{det}(\mathrm{Id}-\mathcal P_A).$$
Next, we use the well-known formula
$$ \mathrm{det}(\mathrm{Id}-\mathcal P_A)=\sum_{m=0}^{n-\kappa}(-1)^m\tr(\Lambda^m\mathcal P_A).$$
To conclude, we show that right hand side in the above equality is equivalent to $(-1)^{d_s}\tr(\Lambda^{d_s}\mathcal P_A)$ when $\|A\|\to +\infty.$ We list the eigenvalues of the Poincare map $\mathcal P_A$ 
\begin{equation}
e^{\lambda_1^-(A)}\leq \ldots \leq e^{\lambda_{d_s}^-(A)}\leq e^{-\eta \|A\|}\leq  e^{\eta \|A\|}\leq e^{\lambda_1^+(A)}\leq \ldots \leq e^{\lambda_{d_u}^+(A)} 
\end{equation}
for some uniform (in $A\in \mathcal C_{aN,bN}$) constant $\eta>0$ given by the Anosov property. The uniformity of the constant comes from the fact that we are working in a proper subcone $\mathcal C$ of the Weyl chamber $\mathcal W$ and that the Anosov constants can only blow up when approaching the boundary of $\mathcal W$. Now we can compute
$$\tr(\Lambda^{k}\mathcal P_{A}):=\sigma_{k}( e^{\lambda_1^-(A)}, \ldots , e^{\lambda_{d_s}^-(A)},e^{\lambda_1^+(A)}, \ldots , e^{\lambda_{d_u}^+(A)})$$
where $\sigma_k$ is the $k$-th symmetric polynomial. We see that the maximum value of $\tr(\Lambda^{k}\mathcal P_{A})$ is attained at $k=d_s$ where one can choose all eigenvalues larger than $1$ without choosing any other eigenvalues. In particular, there is a constant $C>0$, independent of $A\in \mathcal C_{aN,bN}$ and $k\neq d_s$, such that
$$\forall k\neq d_s, \quad |\tr(\Lambda^{k}\mathcal P_{A})|\leq C\tr(\Lambda^{d_s}\mathcal P_{A})e^{-\eta \|A\|}. $$
This means that for any $A\in \mathcal C_{aN,bN}$,
$$\frac{\tr(\Lambda^{d_s}\mathcal P_A)}{|\mathrm{det}(\mathrm{Id}-\mathcal P_A)|}=\frac{1}{(-1)^{d_s}\mathrm{det}(\mathrm{Id}-\mathcal P_A)}\sum_{m=0}^{n-\kappa}(-1)^m\tr(\Lambda^m\mathcal P_A)=1+O(e^{-\eta N a}). $$
Plugging this last estimate into \eqref{eq:deduce} gives \eqref{eq:bowen}.

\textbf{Proof of Corollary \ref{corr}.} 
\begin{proof}
Let $\mathcal C\subset \mathcal W$ be any proper subcone of the Weyl chamber and for $0<a<b$, define $\mathcal C_{a,b}:=\{A\in \mathcal C, \ h_{\mathrm{top}}(\varphi_1^A)/\|h_{\mathrm{top}}^{\mathcal W}\|\in [a,b]\}$. We now consider
$$\mathcal N_{\mathcal C_{a,b}}:=\sum_{T\in \mathcal T}\sum_{A\in L(T)\cap \mathcal C_{a,b}}\mathrm{Vol}(T), \quad \mathcal N^h_{\mathcal C_{a,b}}:=\sum_{T\in \mathcal T}\sum_{A\in L(T)\cap \mathcal C_{a,b}}e^{-h_{\mathrm{top}}(\varphi_1^A)}\mathrm{Vol}(T).$$
For any $q>1$, we use \eqref{eq:bowen} with $f\equiv 1$, $a=1,b=q$ and $N=q^{n-1}$ to obtain
\begin{equation}
\label{eq:exp}
\lim_{n\to+\infty}\frac{\mathcal N^h_{\mathcal C_{q^{n-1},q^n}}}{q^{\kappa(n-1)}}=|\mathcal C_{1,q}|.
\end{equation}
Now, one sees that for any $ \epsilon>0$ and $n\gg 1$, 
$$(|\mathcal C_{1,q}|-\epsilon)q^{\kappa(n-1)}e^{q^{n-1}\|h_{\mathrm{top}}^{\mathcal W}\|} \leq \mathcal N_{\mathcal C_{q^{n-1},q^n}}\leq (|\mathcal C_{1,q}|+\epsilon)q^{\kappa(n-1)}e^{q^{n}\|h_{\mathrm{top}}^{\mathcal W}\|}.$$
Taking the logarithm and using $\mathcal N_{\mathcal C_{0,q^n}} =\mathcal N_{\mathcal C_{0,1}}+\sum_{k=2}^n \mathcal N_{\mathcal C_{q^{k-1},q^k}}$ one gets for $n\gg 1$, 
$$ \|h_{\mathrm{top}}^{\mathcal W}\|/q-\epsilon\leq \frac{\ln(\mathcal N_{\mathcal C_{0,q^n}})}{q^n}\leq \|h_{\mathrm{top}}^{\mathcal W}\|+\epsilon.$$
Since $q$ can be chosen arbitrarily close to $1$, this concludes the proof of the corollary.
\end{proof}

\printbibliography[
heading=bibintoc,
title={References}]

\appendix
\section{Anosov actions for which the measure of maximal entropy and SRB measure do not coincide.}
\label{secappendix}
In this appendix, we show that the construction of Vinhage in \cite{Vin} provides examples of Anosov actions with no rank one factor for which the SRB measure is not equal to the measure of maximal entropy. 

Let us briefly recall the construction. Consider two topologically mixing Anosov flows $\varphi_t:X\to X$ and $\phi_s:Y\to Y$ on two $3$-dimensional closed manifolds $X,Y$. Choose $p_1,p_2\in X$ and $q_1,q_2\in Y$ points on distincts periodic orbits. For $\delta>0$, choose functions $u_1, u_2$ such that
\begin{itemize}
\item $u_1\in C^{\infty}(X)$ and $u_2\in C^{\infty}(Y)$.
\item $|u_1|,|u_2|\leq \delta$.
\item $u_1\equiv \delta$ on the periodic orbit defined by $p_1$ and $u_2\equiv \delta$ on the periodic orbit defined by $q_1$.
\item $u_1\equiv -\delta$ on the periodic orbit defined by $p_2$ and $u_2\equiv -\delta$ on the periodic orbit defined by $q_2$.
\end{itemize}
The functions $u_i$ for $i=1,2$ define cocycles $\theta_i$, that is $\theta_1(t,x)=\int_0^tu_1(\varphi_\tau(x))d\tau$ and $\theta_2(s,y)=\int_0^su_2(\phi_\tau(y))d\tau$. We then define another cocycle by
$$\beta(s,t;x)=(s-\theta_2(t,x_2),t-\theta_1(s,x_1)). $$
Define an action $\alpha :\mathbb R^2\curvearrowright X\times Y$ by
\begin{equation}
\label{eq:alpha}
\alpha(s,t)(x,y)=\big(\varphi_{s-\theta_2(t,y)}(x),\phi_{t-\theta_1(s,x)}(y) \big).
\end{equation}
Then if $\delta>0$ is small enough, \cite[Theorem 5.1]{Vin} shows that $\alpha$ is $C^{\infty}$ Cartan action without rank one factor which is not homogeneous.

 We apply the previous construction to $X=Y=SM$ where $M$ is a closed negatively curved surface which is not hyperbolic. Consider $\phi_t=\varphi_t$ the geodesic flow on $SM$. It is a result of Katok \cite{Ka} that the Liouville measure (which is equal to the SRB measure in this case) is not equal to the measure of maximal entropy (see \cite{Fou,SLVY} for more general statements). From the classification of equilibrium states (see for instance \cite[Theorem 7.3.24]{FishHas}), the unstable jacobian $J^u(x):=-\tfrac{d}{dt}|_{t=0}\mathrm{det}(d(\varphi_t)_{|E_u(x)})$ is not cohomologous to a constant. By the Livsic theorem, this means that there exists two periodic points $v_1,v_2\in SM$ of periods $T_1,T_2>0$ such that
 $$\frac{1}{T_1}\int_{0}^{T_1}J^u(\varphi_t (v_1))dt\neq \frac 1 {T_2}\int_{0}^{T_2}J^u(\varphi_t (v_2))dt. $$
We apply Vinhage's construction to $p_1=q_1=v_1$ and $p_2=q_2=v_2$ and we will denote this action by $\alpha_M$ in the following.
\begin{prop}
The Anosov action $\alpha_M$ is a $C^{\infty}$-Cartan action without rank one factor for which the measure of maximal entropy is not equal to the SRB measure.
\end{prop}
\begin{proof}
We write $\alpha=\alpha_M$ in the proof.
Using \eqref{eq:alpha} and the definition of $u_1,u_2$, one has
\begin{align*}&\alpha(s,s)(v_1,v_1)= \big(\varphi_{s(1-\delta)}(v_1),\phi_{s(1-\delta )}(v_1) \big)
\\&\alpha(s,s)(v_2,v_2)= \big(\varphi_{s(1+\delta)}(v_2),\phi_{s(1+\delta )}(v_2) \big).
\end{align*}
We deduce that $(v_1,v_1)$ is a periodic point of period $\tfrac{T_1}{1-\delta}$ and $(v_2,v_2)$ is a periodic point of period $\tfrac{T_2}{1+\delta}$. Moreover, we see that $\alpha$ is a time change of the product of the two Anosov flows. Since the weak unstable foliation is invariant under time change this means that $E_{cu}^{\alpha}(v_1,v_1)=E_{cu}^{\varphi}(v_1)\oplus E_{cu}^{\varphi}(v_1)$. In particular, the unstable Jacobian is given by
\begin{align*}J^u_{\alpha(1,1)}(\alpha(s,s)(v_1,v_1))&=-\frac{d}{dt}|_{t=0}\mathrm{det}(d\alpha(s+t,s+t)(v_1,v_1)|_{E_{cu}(\varphi_s(v_1))\oplus E_{cu}(\varphi_s(v_1))})
\\&=2(1-\delta)J^u(\varphi_sv_1), 
\end{align*}
where we used that the center direction was isometric to add it into the definition of the unstable Jacobian. Similarly, 
$$J^u_{\alpha(1,1)}(\alpha(s,s)(v_2,v_2))=2(1+\delta)J^u(\varphi_sv_2). $$ In particular, integrating on the periodic orbits yields
\begin{align*}\frac{1-\delta}{T_1}\int_{0}^{\tfrac{T_1}{1-\delta}}J^u_{\alpha(1,1)}(\alpha(s,s)(v_1,v_1))ds&=\frac{2}{T_1}\int_{0}^{T_1}J^u(\varphi_t (v_1))dt\neq \frac 2 {T_2}\int_{0}^{T_2}J^u(\varphi_t (v_2))dt
\\&= \frac{1+\delta}{T_2}\int_{0}^{\tfrac{T_2}{1+\delta}}J^u_{\alpha(1,1)}(\alpha(s,s)(v_2,v_2))ds.
\end{align*}
In other words, we showed that the unstable Jacobian $J^u_{\alpha(1,1)}$ is not cohomologous to a constant. We can now show that this implies that the SRB measure is not the measure of maximal entropy. We follow the proof of the classification of equilibrium measures (see \cite[Theorem 7.3.24]{FishHas}) and use the Gibbs property. Apply \cite[Theorem 5.1, point 3]{CPZ20} to $\alpha(T_1',T_1')$ where $T_1'=(1-\delta)T_1$ and to $(v_1,v_1)$. Then for $\epsilon>0$, there are constants $a(\epsilon),b(\epsilon)>0$ such that
$$a(\epsilon)\leq \frac{m\big(B_n^{\alpha(T_1',T_1')}((v_1,v_1),\epsilon)\big)}{e^{-nh_{\mathrm{top}}(\alpha(T_1',T_1'))}}\leq b(\epsilon) $$
where $m$ is the measure of maximal entropy and $B_n^{f}(x,\epsilon)$ denotes the Bowen ball for a partially hyperbolic map $f$, see \eqref{eq:BB}. Now, using \cite[Theorem 5.2, point 3]{CPZ20} (the unstable Jacobian satisfies the $u$-Bowen property by \cite[Proof of Theorem 5.2]{CPZ20} and $cs$-Bowen property by \cite[Lemma 3.17]{CarHer}, see \cite[Section 4]{CPZ20} for the precise definitions)  we get constants $a'(\epsilon),b'(\epsilon)>0$ such that
$$a'(\epsilon)\leq \frac{\mu\big(B_n^{\alpha(T_1',T_1')}((v_1,v_1),\epsilon)\big)}{e^{-n\int_0^{1}J^u_{\alpha(T_1',T_1')}[\alpha(T_1's,T_1's)(v_1,v_1)]ds}}\leq b'(\epsilon), $$
where $\mu$ is the SRB measure. If we have $\mu=m$, then combining the previous two Gibbs bound and taking $n\to+\infty$ gives 
$$\int_0^{1}J^u_{\alpha(T_1',T_1')}[\alpha(T_1's,T_1's)(v_1,v_1)]ds= h_{\mathrm{top}}(\alpha(T_1',T_1')).$$
This is equivalent to
$$\frac 1{T_1'}\int_0^{T_1'}J^u_{\alpha(1,1)}(\alpha(s,s)(v_1,v_1))ds=\frac{1}{T_1'}h_{\mathrm{top}}(\alpha(T_1',T_1'))= h_{\mathrm{top}}(\alpha(1,1)).$$
Applying the same argument to $v_2$ gives (we write $T_2'=T_2(1+\delta)$)
$$ \frac 1{T_2'}\int_0^{T_2'}J^u_{\alpha(1,1)}(\alpha(s,s)(v_2,v_2))ds=h_{\mathrm{top}}(\alpha(1,1))=\frac 1{T_1'}\int_0^{T_1'}J^u_{\alpha(1,1)}(\alpha(s,s)(v_1,v_1))ds$$
which is a contradiction. This concludes the proof of the proposition.
\end{proof}

\end{document}